\documentclass[graybox]{svmult}

\usepackage{type1cm}        
\usepackage{makeidx}         
\usepackage{graphicx}        
\usepackage{multicol}        
\usepackage[bottom]{footmisc}

\usepackage{pgf,tikz,pgfplots}
\usepackage{newtxtext}       %
\usepackage[varvw]{newtxmath}       

\newcommand{\MD}{\mathcal{D}}

\newcommand{\bCOP}{\partial\text{\rm COP}}

\newcommand{\kzeroc}{\MK^{(0)}\text{-certificate}}

\newcommand{\oN}{{\mathbb N}}
\newcommand{\oR}{{\mathbb R}}

\newcommand\norm[1]{\left\lVert#1\right\rVert}

\newcommand{\ignore}[1]{}

\newcommand{\supp}{\text{\rm Supp}}
\newcommand{\COP}{{\text{\rm COP}}}
\newcommand{\MK}{{\mathcal K}}
\newcommand{\LAS}{{\text{\rm LAS}}}
\newcommand{\MQ}{{\mathcal Q}}
\newcommand{\MS}{{\mathcal S}}
\newcommand{\MC}{{\mathcal C}}
\newcommand{\MM}{{\mathcal M}}
\newcommand{\MT}{{\mathcal T}}

\newcommand{\SPN}{{\text {\rm SPN}}}

\newcommand{\mmod}{\text{\rm mod }}
\newcommand{\las}{\text{\rm las}}

\makeindex             

\begin{document}

\title*{Copositive matrices, sums of squares and the stability number of a graph}
\titlerunning{Copositive matrices, sums of squares and stable sets in graphs}
\author{Luis Felipe Vargas and Monique Laurent}
\institute{Luis Felipe Vargas \at  Centrum Wiskunde \& Informatica,  Science Park 123
1098 XG Amsterdam, Netherlands, \email{luis.vargas@cwi.nl}
\and Monique Laurent \at Centrum Wiskunde \& Informatica, Science Park 123
1098 XG Amsterdam, Netherlands, and Tilburg University, Warandelaan 2, 5037 AB Tilburg, Netherlands, \email{m.laurent@cwi.nl}
\\
\\
This work is supported by the European Union's Framework Programme for Research and Innovation Horizon
2020 under the Marie Skłodowska-Curie Actions Grant Agreement No. 813211  (POEMA).}
%
%
\maketitle

\abstract*{Each chapter should be preceded by an abstract (no more than 200 words) that summarizes the content. The abstract will appear \textit{online} at \url{www.SpringerLink.com} and be available with unrestricted access. This allows unregistered users to read the abstract as a teaser for the complete chapter.
Please use the 'starred' version of the \texttt{abstract} command for typesetting the text of the online abstracts (cf. source file of this chapter template \texttt{abstract}) and include them with the source files of your manuscript. Use the plain \texttt{abstract} command if the abstract is also to appear in the printed version of the book.}

\abstract{This chapter investigates the cone of copositive matrices, with a focus on the design and analysis of conic inner approximations for it. These approximations are based on various sufficient conditions for matrix copositivity, relying on positivity certificates in terms of sums of squares of polynomials. Their application to the discrete optimization problem asking for a maximum stable set in a graph is also discussed. A central theme in this chapter is understanding when the conic approximations suffice for describing the full copositive cone, and when the corresponding bounds for the stable set problem admit finite convergence.
}

\section{Introduction} \label{sec:1}

An $n\times n$ symmetric matrix $M$ is said to be {\em copositive} if the associated quadratic form $x^TMx=\sum_{i,j=1}^nM_{ij}x_ix_j$ is nonnegative over the nonnegative orthant $\oR^n_+$.
The set of copositive matrices is a cone, the {\em copositive cone} $\COP_n$, thus defined as 
\begin{align}\label{def-COP}
\COP_n=\{M \in \MS^n : x^TMx\geq 0 \quad \forall x\in \mathbb{R}_+^n\}.
\end{align}
Copositive matrices are a fundamental class of matrices that play an important role in several areas, including linear algebra and combinatorial matrix theory (see the monograph \cite{book-Shaked-Berman}) and optimization
(see, e.g., the overview \cite{Dur-Survey}). Their relevance in optimization is illustrated by the fact that many hard combinatorial optimization problems
can be formulated as linear optimization problems over the copositive cone (see, e.g.,  \cite{Betal,Burer2009,dKP2002,DR10,GL2008}). This is the case, in particular, for the problem of determining the maximum stable set in a graph, a topic that we will discuss in this chapter (see Section \ref{section-alpha}).

Hence the copositive cone has a broad modeling power. As a consequence it is a computationally hard object to work with: linear optimization over $\COP_n$ is an NP-hard problem and checking 
 whether a matrix is copositive is a co-NP-complete problem \cite{MK1987}. Motivated by these hardness results,  several hierarchies of conic inner approximations for $\COP_n$ have been introduced in the literature. A key ingredient in these approximations 
 is to design tractable certificates 
  that permit to certify that the quadratic form $x^TMx$ is nonnegative over $\oR^n_+$ and thus that the matrix $M$ is copositive. 
These certificates are based on using sums of squares of polynomials as a ``proxy" for global nonnegativity, which is motivated by  the fact that sums of squares of polynomials can be modeled using semidefinite optimization (as recalled later in relation (\ref{sossdp})).
  
 Another possible approach to certify copositivity of a matrix $M$ is to consider the quartic form 
\begin{align}\label{eqquartic}
(x^{\circ 2})^TMx^{\circ 2}:=\sum_{i,j=1}^nM_{ij}x_i^2x_j^2
\end{align}
 and to design sum-of-squares based  certificates that certify that  $(x^{\circ 2})^TMx^{\circ 2} $ is nonnegative on the full space $\oR^n$. In other words, one may rely on the following alternative definition of the copositive cone
 \begin{align}\label{def-COP-sq}
\COP_n=\{M\in \MS^n: (x^{\circ 2})^TMx^{\circ 2}\geq 0 \text{ for all } x\in \mathbb{R}^n\},
\end{align} 
where we let $x^{\circ 2}=(x_1^2,\ldots,x_n^2)$ denote the vector of squared variables.  

As we will see in this chapter, these two equivalent definitions (\ref{def-COP}) and (\ref{def-COP-sq}) of the copositive cone offer the starting point for the definition of several hierarchies of conic approximations. Our objective in this chapter is to discuss  the relationships between these various hierarchies, their convergence properties, and their application to the maximum stable set problem in graphs.  We now briefly describe the contents of this chapter.

\subsection*{Organization of the chapter}
 In Section \ref{preliminaries} we introduce some general background about polynomial optimization and  sums of squares of polynomials.  In particular, in Section \ref{seccertificate},  we recall some important positivity certificates  that permit to certify the nonnegativity of a polynomial on  the nonnegative orthant and on compact semialgebraic sets.   In Section \ref{secpop} we describe how these positivity certificates are used to define hierarchies of bounds for polynomial optimization problems and, in Section \ref{secNie}, we recall  a criterion that can be used to detect when the bounds have finite convergence.

\medskip
In Section \ref{secconic} we present several hierarchies of conic inner approximations for the copositive cone $\COP_n$. 
These conic approximations  are based on using different types of positivity certificates for the quadratic form $x^TMx$,  or for  the quartic form $(x^{\circ 2})^TMx^{\circ 2}$ from (\ref{eqquartic}). Moreover, one considers positivity on the full space $\oR^n$, on the nonnegative orthant $\oR^n_+$, on the standard simplex $\Delta_n=\{x\in \oR^n_+: \sum_{i=1}^n x_i=1\}$, or on the unit sphere 
$\mathbb S^{n-1}=\{x\in \oR^n: \sum_{i=1}^n x_i^2=1\}$.

In Section~\ref{secconeK} we introduce the cones $\MC^{(r)}_n$ and $\MK^{(r)}_n$, where, for $\MC^{(r)}_n$, one requires that   the polynomial 
$(\sum_{i=1}^nx_i)^rx^TMx$ has nonnegative coefficients, and, for $\MK^{(r)}_n$, one requires that 
the polynomial $(\sum_{i=1}^nx_i^2)^r (x^{\circ 2})^TMx^{\circ 2}$ is a sum of squares of polynomials. These two conic hierarchies are motivated by the representation results by Reznick (for positive polynomials on $\oR^n$, Theorem \ref{theorem-reznick}) and by P\'olya (for positive polynomials on $\oR^n_+$, Theorem \ref{theorem-polya}).
In addition, the cones $\MQ^{(r)}_n$ are introduced as a simpler, but weaker variation of the cones $\MK^{(r)}_n$.

In Section~\ref{secconeLAS} we introduce the Lasserre-type cones $\LAS^{(r)}_{\Delta_n}$, $\LAS^{(r)}_{\Delta_n,\MT}$ and $\LAS^{(r)}_{\mathbb S^{n-1}}$, where, respectively, one now uses positivity certificates for the polynomial $x^TMx$ on the standard simplex $\Delta_n$ (using representations in the quadratic module or the preordering of $\Delta_n$),  and positivity certificates for the polynomial $(x^{\circ 2})^TMx^{\circ 2}$ on the unit sphere $\mathbb S^{n-1}$. The motivation for these cones now stems from the representation results by Schm\"udgen (Theorem \ref{theorem-schmudgen}) and by Putinar (Theorem \ref{theorem-putinar}).

In Section~\ref{secconelinks} we explain in detail the relationships  between these various hierarchies of conic approximations of the copositive cone (see Theorem \ref{theo-link}).

Each of the above hierarchies of conic approximations covers the {\em interior} of the copositive cone, which follows from the above mentioned representation results. This raises naturally the question of whether some of these hierarchies are  able to cover the {\em full} copositive cone (i.e., also its boundary). This question is the central theme of Section \ref{secexact}.

\medskip Section \ref{secexact} is devoted to investigating exactness properties of the above  hierarchies of cones, i.e., for which matrix sizes   the hierarchies are able to cover the full copositive cone $\COP_n$. 
This question is studied for the cones $\MK^{(r)}_n$ in Section \ref{secexactK} and  for the cones $\LAS^{(r)}_{\Delta_n}$ in Section~\ref{section-membership-LAS}. Section~\ref{secCOP5} is devoted to the exceptional case $n=5$, where one can show that   the hierarchy of cones $\MK^{(r)}_5$ covers the full copositive cone $\COP_5$.

\medskip
Section \ref{section-alpha} discusses the application of the various conic approximation hierarchies for $\COP_n$ to the design of upper bounds for the graph parameter $\alpha(G)$, defined as the maximum cardinality of a stable set in a graph $G$. 
In particular, the cones $\MC^{(r)}_n$ lead to the linear programming based parameters  $\zeta^{(r)}(G)$, discussed in Section~\ref{sec-alpha-C}, and the cones  $\MK^{(r)}_n$ lead to the semidefinite bounds $\vartheta^{(r)}(G)$, discussed in   Section~\ref{sec-alpha-K}.
The main theme in this section is to investigate whether the parameters $\vartheta^{(r)}(G)$ do admit finite convergence to $\alpha(G)$ or, equivalently, whether a class of associated copositive  matrices $M_G$ belong to the union $\bigcup_r\MK^{(r)}_n$. This question, which relates to a long standing conjecture by de Klerk and Pasechnik \cite{dKP2002}, is now settled in the affirmative and a sketch of proof is offered in this section. 

\medskip
We conclude with some observations and further research directions in the last Section \ref{secfinal}.
 
\subsubsection*{Notation}
Throughout we will use the following notation. 
For $n\in \mathbb{N}$ we set $[n]=\{1,2, \dots, n\}$. The nonnegative orthant is $\oR^n_+=\{x\in \oR^n: x_1,\ldots,x_n\ge 0\}$, 
the standard simplex in $\oR^n$ is defined as $\Delta_n=\{x\in \mathbb{R}^n_+: 
\sum_{i=1}^nx_i=1\}$, and the unit sphere in $\mathbb{R}^n$ is defined as $\mathbb{S}^{n-1}=\{x\in \mathbb{R}^n: \sum_{i=1}^nx_i^2=1\}$.
For $x\in \oR^n$, the support of $x$ is the set $\{i\in [n]: x_i\ne 0\}$ and we let $x^{\circ 2}:=(x_1^2,\ldots,x_n^2)$ denote the vector of squared entries. We use the notation $e$ to denote the all-ones vector (of appropriate size), so $e=(1,\ldots,1)^T$. For a sequence $\alpha\in \oN^n$, we set $|\alpha|:=\sum_{i=1}^n\alpha_i$.

Throughout, $\MS^n$ denotes the set of $n\times n$ symmetric matrices. We say that a matrix $M\in \MS^n$ is positive semidefinite (denoted as $M\succeq 0$) if $x^TMx\geq 0$ for all $x\in \mathbb{R}^n$. The set of $n\times n$ positive semidefinite matrices is denoted by $\MS_+^n$. The set of diagonal matrices with strictly positive diagonal entries is denoted by $\MD^n_{++}$. We let $I_n, J_n$ (or simply $I,J$) denote the identity matrix and the all-ones matrix in $\MS^n$.

We denote by $\mathbb{R}[x_1, x_2, \dots, x_n]$ the set of polynomials with real coefficients in $n$ variables. Throughout we abbreviate $\mathbb{R}[x_1, \dots, x_n]$ by  $\mathbb{R}[x]$ when there is no ambiguity. Any polynomial is of the form $p=\sum_{\alpha\in \oN^n} p_\alpha x^\alpha$, where only finitely many coefficients $p_\alpha$ are nonzero. Then $|\alpha|$ is the degree of the monomial $x^\alpha=x_1^{\alpha_1}\cdots x_n^{\alpha_n}$ and the degree of $p$, denoted $\deg(p)$, is the maximum degree of its terms $p_\alpha x^\alpha$ with $p_\alpha\ne 0$. We denote by $\mathbb{R}[x]_r$ the set of polynomials of degree at most $r$.  A form, also known as  a homogeneous polynomial, is a polynomial in which all its terms have the same degree. 

Given a polynomial $f\in \oR[x]$ and a set $K\subseteq \oR^n$, we say that  $f$ is {\em nonnegative} (or {\em positive}) on the set $K$ if $f(x)\geq0$ for all $x\in K$, and we say that $f$ is {\em strictly positive} on $K$ if $f(x)>0$ for all $x\in K$.  Given a tuple of polynomials $h=(h_1, \dots, h_l)$, the ideal generated by $h$ is defined as $I(h):=\{\sum_{i=1}^lq_ih_i: q_i\in \mathbb{R}[x]\}$. Its truncation at degree $r$ is defined as $I(h)_r:=\{\sum_{i=1}^lq_ih_i: \text{deg}(q_ih_i)\leq r \text{ for } i\in [m]\}$. 
We will in particular consider the case when $h=\sum_{i=1}^nx_i-1$ or $h=\sum_{i=1}^nx_i^2-1$, that define the simplex $\Delta_n$ and the unit sphere $\mathbb S^{n-1}$, respectively.
Then we use the shorthand notation $I_{\Delta_n}:=I(\sum_{i=1}^nx_i-1)$ and 
 $I_{\mathbb{S}^{n-1}}:=I(\sum_{i=1}^nx_i^2-1)$. Finally, we let $\Sigma:=\{\sum_{i=1}^mq_i^2: q_i\in \mathbb{R}[x]\}$ denote the cone of sums of squares of polynomials, and, for an integer $r\in\oN$,  $\Sigma_r=\Sigma\cap \mathbb{R}[x]_r$ is the subcone consisting of  the sums of squares that have degree at most $r$.


\section{Preliminaries on polynomial optimization, nonnegative polynomials and sums of squares}\label{preliminaries}
Polynomial optimization asks for minimizing a polynomial over a semialgebraic set. That is, given polynomials $f, g_1, \dots, g_m, h_1, \dots, h_l \in \mathbb{R}[x]$,  the task is to find (or approximate) the infimum of the following problem
\begin{align}\label{poly-opt}
f^*=\inf_{x\in K} f(x) ,
\end{align}
where 
\begin{align}\label{K}
K=\big\{x\in \mathbb{R}^n: g_i(x)\geq 0 \text{ for } i=1, \dots, m \text{ and } h_i(x)=0 \text{ for } i=1,\dots, l\big\}
\end{align}
is a semialgebraic set. Problem (\ref{poly-opt}) can be equivalently rewritten as 
\begin{align}\label{poly-opt-lambda}
f^*=\sup \{\lambda: f(x)-\lambda \geq 0 \text{ for all } x\in K\}.
\end{align}
In view of this new formulation, finding lower bounds for a polynomial optimization problem amounts to finding certificates that certain polynomials are nonnegative on the semialgebraic set $K$.

\subsection{Sum-of-squares certificates for nonnegativity}\label{seccertificate}

Testing whether a polynomial is nonnegative on a semialgebraic set is hard in general. Even testing whether a polynomial is globally nonnegative (nonnegative on $K=\mathbb{R}^n$) is a hard task in general. An easy \textit{sufficient} condition for a polynomial to be globally nonnegative is being a sum of squares. A polynomial $p\in \mathbb{R}[x]$ is said to be a {\em sum of squares} if it can be written as a sum of squares of other polynomials, i.e., if $p=q_1^2 + \dots +q_m^2$ for some $q_1,\ldots,q_m\in \mathbb{R}[x]$. Hilbert \cite{Hilbert1, Hilbert2} showed that every nonnegative polynomial of degree $2d$ in $n$ variables is a sum of squares in the following cases: $(2d, n)$=$(2d,1)$, $(2,n)$, or $(4,2)$. Moreover, he showed that for any other pair $(2d,n)$ there exist nonnegative polynomials that are not sums of squares. The first explicit example of a nonnegative polynomial that is not a sum of squares was given by Motzkin \cite{motzkin-p} in 1967.

\begin{example}{The Motzkin polynomial is nonnegative, but not a sum of squares}\label{example-motzkin-1}
The following polynomial in two variables is known as the Motzkin polynomial:
\begin{align}\label{motzkin-p}
h(x,y)=x^4y^2+x^2y^4-3x^2y^2+1.
\end{align}
The Motzkin polynomial is nonnegative in $\mathbb{R}^2$. 
This can be seen, e.g.,  by using the Arithmetic-Geometric Mean inequality, which gives
$$\frac{x^4y^2+x^2y^4+1}{3}\geq \sqrt[3]{x^4y^2\cdot x^2y^4 \cdot 1}=x^2y^2.$$
However, $h(x,y)$ cannot be written as a sum of squares. This can be checked using ``brute force":  assume $h =\sum_i q_i^2$ and examine the coefficients on both sides (starting from the coefficients of the monomials $x^6,y^6$, etc.; see, e.g.,  \cite{Reznick2000}). 
\\
The Motzkin form is the homogenization of $h$, thus the  homogeneous  polynomial in three variables:
\begin{align}\label{motzkin-f}
m(x,y,z)=x^4y^2+x^2y^4-3x^2y^2z^2+z^6.
\end{align}
Hence, the Motzkin form is  nonnegative on $\oR^3$ and it cannot be written as a sum of squares.
\end{example}
In 1927 Artin \cite{Artin} proved that any globally nonnegative polynomial $f$ can be written as a sum of squares of rational functions, i.e., $f=\sum_{i}(\frac{p_i}{q_i})^2$ for some $p_i,q_i\in \mathbb{R}[x]$, solving affirmatively  Hilbert's 17th problem. Equivalently, Artin's result shows that for any nonnegative polynomial $f$ there exists a polynomial $q$ such that  $q^2f\in \Sigma$. Such certificates are sometimes referred to as certificates ``with denominator". The following result shows that, when $f$ is homogeneous and strictly positive on $\mathbb{R}^n\setminus\{0\}$, the denominator can be chosen to be a power of $(\sum_{i=1}^nx_i^2)$.
\begin{theorem}[Reznick \cite{Reznick1995}]\label{theorem-reznick}
Let $f\in \mathbb{R}[x]$ be a homogeneous polynomial such that $f(x)>0$ for all $x\in \mathbb{R}^n\setminus \{0\}$. Then the following holds:
\begin{align}\label{cert-reznick}
\Big(\sum_{i=1}^nx_i^2\Big)^rf\in \Sigma \quad \text{ for some $r\in \mathbb{N}$}.
\end{align}
\end{theorem}
Scheiderer \cite{Scheiderer} shows that the {\em strict} positivity condition can be omitted for $n=3$: any {\em{nonnegative}} form $f$ in three variables admits a certificate as in (\ref{cert-reznick}). On the negative side, this is not the case  for  $n\geq 4$: there exist nonnegative forms in $n\ge 4$ variables that do not admit a positivity certificate as in (\ref{cert-reznick}) (an example is given below).

\begin{example}{Certificate for nonnegativity of the Motzkin polynomial}
Let  $h(x,y)=x^4y^2+x^2y^4-3x^2y^2+1$ be the Motzkin polynomial, which  is nonnegative and not a sum of squares. However, 
$$(x^2+y^2)^2h(x,y)=x^2y^2(x^2+y^2+1)(x^2+y^2-2)^2+(x^2-y^2)^2$$
is a sum of squares. This sum-of-squares certificate  thus shows  (again) that $h$ is nonnegative on $\mathbb{R}^2$.
\end{example}

\begin{example}{A nonnegative polynomial $f$ such that $(\sum_{i=1}^nx_i^2)^rf\notin \Sigma$ for all $r\in \mathbb{N}$}
Let $q(x,y,z,w):=m^2+w^6m$, where $m$ is the Motzkin form from (\ref{motzkin-f}). Clearly, $q$ is nonnegative on $\mathbb{R}^4$, as $m$ is nonnegative on $\mathbb{R}^3$. Assume that  there exists $r\in\oN$ such that $(x^2+y^2+z^2+w^2)^rq\in \Sigma$. Then, $p':=(x^2+y^2+z^2+1)^rq(x,y,z,1)=(x^2+y^2+z^2+1)^r(m^2+m)$ is also a sum of squares. 
As $p'$ is a sum of squares, one can check that also  its lowest degree homogeneous part  is a sum of squares (see \cite[Lemma 4]{LV2021b}). However, the lowest degree homogeneous part of $p'$ is $m$, which is not a sum of squares. Hence this shows that  $(x^2+y^2+z^2+w^2)^rq\not\in \Sigma$ for all $r\in \oN$.
\end{example}

Next, we give some positivity certificates for polynomials on semialgebraic sets. The following result shows the existence of a positivity certificate for polynomials that are strictly positive on the nonnegative orthant $\mathbb{R}_+^n$.
\begin{theorem}[P\'{o}lya \cite{Polya}]\label{theorem-polya}
Let $f$ be a homogeneous polynomial such that $f(x)>0$ for all $x\in \mathbb{R}_+^n\setminus{\{0\}}$. Then the following holds:
\begin{align}\label{cert-polya}
\Big(\sum_{i=1}^nx_i\Big)^rf \text{ has nonnegative coefficients \quad for some } r\in \mathbb{N}.
\end{align}
\end{theorem}
In addition, Castle, Powers, and Reznick \cite{CPR} show that nonnegative polynomials on $\mathbb{R}^n_+$ with finitely many zeros (satisfying some technical properties) also admit a certificate as in (\ref{cert-polya}).

Now we consider positivity  certificates for polynomials restricted to compact semialgebraic sets. Let $g=\{g_1, \dots, g_m\}$ and $h=\{h_1, \dots, h_l\}$ be sets of polynomials and consider the semialgebraic set $K$ defined as in 
(\ref{K}). 
The \textit{quadratic module} generated by $g$, denoted by $\MM(g)$, is defined as
\begin{align}
\MM(g):=\Big\{\sum_{i=0}^m\sigma_ig_i  :  \sigma_i\in \Sigma \text{ for } i=0,1, \dots, m, \text{ and } g_0:=1\Big\},
\end{align} 
and the \textit{preordering} generated by $g$, denoted by $\MT(g)$, is defined as
\begin{align}
\MT(g):=\Big\{\sum_{J\subseteq [m]}\sigma_J\prod_{i\in J}g_i: \sigma_J\in \Sigma \text{ for } J\subseteq \{1,\dots m\}, \text{ and } g_{\emptyset}:=1\Big\}.
\end{align}  
Observe that, if for a polynomial $f$ we have
\begin{align}
f \in\MM(g) + I(h), \label{cert-putinar}\\
 \text{ or } \
f\in \MT(g) + I(h),\label{cert-schmudgen}
\end{align}
then $f$ is nonnegative on $K$. Moreover, if a polynomial admits a certificate as in 
(\ref{cert-putinar}), then it also admits a certificate as in (\ref{cert-schmudgen}), because $\MM(g)\subseteq \MT(g)$.

\ignore{
\begin{align*}
\mathcal{M}(K)&=\Big\{\sum_{i=0}^n\sigma_ig_i +  q : \sigma_i\in \Sigma, q\in I_h\}, \quad \text{where }g_0:=1\Big\},
\\
\mathcal{T}(K)&= \Big\{\sum_{I\subseteq [m]}\sigma_Ig_I + q : \sigma_J\in \Sigma, q\in I_h\Big\}, \quad \text{where } g_I:=\prod_{i\in J}g_i.
\end{align*}
Clearly,  $\mathcal{M}(K)\subseteq \mathcal{T}(K)$. If a polynomial belongs to  $\mathcal{M}(K)$ or  to $\mathcal{T}(K)$ then it is clearly nonnegative on K.
}

\begin{example}{Example}
Consider the polynomial $p(x,y)=x^2+y^2-xy$ in two variables $x,y$. We show that $p$ is nonnegative on $\mathbb{R}^2_+$
 in two different ways. The following identities hold:
 \begin{align*}
 (x+y)p(x,y)=x^3+y^3, \\
 p(x,y)=(x-y)^2+xy, 
 \end{align*}
which both 
certify that $p$ is nonnegative on $\mathbb{R}_+^2$. The first identity is a certificate as in (\ref{cert-polya}):  $x^3+y^3$ has nonnegative coefficients. The second identity shows that $p\in \MT(\{x,y\})$, i.e.,  gives a certificate as in (\ref{cert-schmudgen}).
 \end{example}
 
The following two theorems show that, under certain conditions on the semialgebraic set $K$ (and on the tuples $g$ and $h$ defining it), every strictly positive polynomial admits certificates as in (\ref{cert-putinar}) or (\ref{cert-schmudgen}).

\begin{theorem}[Schm\"{u}dgen  \cite{Schmudgen}]\label{theorem-schmudgen}
Let $K=\{x\in \mathbb{R}^n: g_i(x)\geq 0 \text{ for }i\in[m], h_j(x)=0\text{ for }j\in[l]\}$ be a compact semialgebraic set. Let $f\in \mathbb{R}[x]$ such that $f(x)>0$ for all $x\in K$. Then we have $f\in \mathcal{T}(g)+I(h)$.
\end{theorem}
We say that the sets of polynomials $g=\{g_1, \dots, g_m\}$ and $h=\{h_1, \dots, h_l\}$ satisfy the \textit{Archimedean condition} if 
\begin{align}\label{archimedean-cond}
N-\sum_{i=1}^n x_i^2\in \MM(g) + I(h) \quad \text{for some } N\in \mathbb{N}.
\end{align}
Note this implies that the associated set $K$ is compact. We have the following result.

\begin{theorem}[Putinar \cite{Putinar}]\label{theorem-putinar}
Let $K=\{x\in \mathbb{R}^n: g_i(x)\geq 0 \text{ for }i\in[m], h_j(x)=0\text{ for }j\in[l]\}$ be a semialgebraic set. Assume the sets of polynomials $g=\{g_1, \dots, g_m\}$ and $h=\{h_1, \dots, h_l\}$ satisfy the Archimedean condition (\ref{archimedean-cond}). Let $f\in \mathbb{R}[x]$ be such that $f(x)>0$ for all $x\in K$. Then we have $f\in \mathcal{M}(g)+I(h)$.
\end{theorem}
Note that positivity certificates for a polynomial $f$ as in Theorem \ref{theorem-schmudgen} and Theorem \ref{theorem-putinar} involve a representation of the polynomial $f$ ``without denominators".

\subsection{Approximation hierarchies for polynomial optimization}\label{secpop}
Based on the result in Putinar's theorem, Lasserre \cite{Lasserre1} proposed a hierarchy of approximations $(f^{(r)})_{r\in \oN}$ for problem (\ref{poly-opt}). Given an  integer $r\in \mathbb{N}$,   the \textit{quadratic module truncated at degree $r$} (generated by the set $g=\{g_1, \dots, g_m\}$)   is defined as 
\begin{align}\label{truncated-qm}
\mathcal{M}(g)_{r}:=\Big\{\sum_{i=0}^m\sigma_ig_i  : \sigma_i\in \Sigma_{r-\text{deg}(g_i)}\ \text{\rm for } i\in\{0,1,\ldots,m\}, \text{\rm and }  g_0=1\Big\},
\end{align}
and  the  parameter $f^{(r)}$  as
\begin{align}\label{lasserre-hierarchy}
f^{(r)}:= \sup \{\lambda : f-\lambda\in \MM(g)_r + I(h)_r\}.
\end{align}
Clearly, $f^{(r)}\le f^{(r+1)} \leq f^*$ for all $r \in \mathbb{N}$. The hierarchy of parameters $f^{(r)}$ is also known as \textit{Lasserre sum-of-squares hierarchy} for problem (\ref{poly-opt}).

\begin{example}{Semidefinite programming and sums of squares}
Consider a polynomial $p\in \mathbb{R}[x]_{2d}$. The following observation was made in \cite{CLR}: 
\begin{align}\label{sossdp}
p\in \Sigma_{2d} \Longleftrightarrow p=[x]_{d}^TM[x]_{d} \text{ for some } M\succeq 0,
\end{align}
where $[x]_d=(x^{\alpha})_{|\alpha|\leq d}$ denotes the vector of monomials with degree at most $d$.

 Indeed, if $p\in \Sigma_{2d}$ then $p=\sum_{i=1}^m q_i^2$ for some $q_i\in \mathbb{R}[x]_d$. We can write $q_i=[x]_d^Tv_i$ for an appropriate vector $v_i$. Then, we obtain $p=\sum_{i=1}^m q_i^2 = [x]_d^T(\sum_{i=1}^mv_iv_i^T)[x]_d^T=[x]_d^TM[x]_d$, where $M:=\sum_{i=1}^mv_iv_i^T$ is a positive semidefinite matrix.

Conversely, assume $p=[x]^T_dM[x]_d$ with $M\succeq 0$. Then $M=\sum_{i=1}^mv_iv_i^T$ for some vectors $v_1, \dots, v_m$. Hence, $p=\sum_{i=1}^m([x]_d^Tv_i)^2$ is a sum of squares. 

So relation (\ref{sossdp}) shows that testing whether a given polynomial is a sum of squares can be modeled as  a semidefinite program. There exist efficient algorithms for solving semidefinite programs (up to any arbitrary precision, and under some technical assumptions). See, e.g., \cite{BTN,dKbook}.
\end{example}

Under the Archimedean condition, by Putinar's theorem,  we have asymptotic convergence of the Lasserre hierarchy: $f^{(r)}\to f^*$ as $r\to \infty$. We say that {\em finite convergence} holds if $f^{(r)}=f^*$ for some $r\in \mathbb{N}$. In general, finite convergence does not hold, as the following example shows. 

\begin{example}{A polynomial optimization problem without finite convergence}\label{example-finite}
Consider the problem
$$\min \quad x_1x_2 \quad \text{ s.t. }\quad  x\in \Delta_3, \text{ \rm i.e., } x_1\geq 0, x_2\geq 0, x_3\geq 0, x_1+x_2+x_3=1.$$
We show that the Lasserre hierarchy for this problem does not have finite convergence. The optimal value is clearly 0 and is attained, for example, in $x=(0,0,1)$. Assume the Lasserre hierarchy has finite convergence.  Then, 
\begin{align}\label{no-finite}
x_1x_2= \sigma_0 + \sum_{i=1}^{3} x_i\sigma_i +q(\sum_{i=1}^{3}x_i-1),
\end{align}
for some $\sigma_i\in \Sigma$ for $i=0,1,2,3$ and $q\in \mathbb{R}[x]$. For a scalar $t\in (0,1)$ define the vector $u_t:=(t,0,1-t)\in \Delta_3$. Now we evaluate equation (\ref{no-finite}) at $x+u_t$ and obtain 
\begin{align*}
x_1x_2 + tx_2 = \sigma_0(x+u_t) + (x_1+t)\sigma_1(x+u_t) + x_2\sigma_2(x+u_t)\\
+(x_3+1-t)\sigma_3(x+u_t) + q(x+u_t)(x_1+x_2+x_3).
\end{align*}
for any fixed $t\in (0,1)$. We compare the coefficients of the polynomials in $x$ at both sides of the above identity.
Observe that there is no constant term in the left hand side, so $\sigma_0(u_t) + t\sigma_1(u_t) + (1-t)\sigma_3(u_t)=0$, which implies $\sigma_i(u_t)=0$ for $i=0,1,3$ as $\sigma_i\in \Sigma$ and thus $\sigma_i(u_t)\geq 0$. Then, for $i=0,1,3$, the polynomial $\sigma_i(x+u_t)$ has no constant term, and thus  it has no linear terms.  Now, by comparing the coefficient of $x_1$ at both sides, we get $q(u_t)=0$. Finally, by comparing the coefficient of $x_2$ at both sides, we get $t=\sigma_2(u_t)$ for all $t\in (0,1)$. This implies  $\sigma_2(u_t)=t$ as polynomials in the variable $t$. This is a contradiction because $\sigma_2(u_t)$ is a sum of squares in $t$.
\end{example}

\subsection{Optimality conditions and finite convergence}\label{secNie}
In this section we recall a result of Nie \cite{Nie} that guarantees finite convergence of the Lasserre hierarchy (\ref{lasserre-hierarchy}) under some assumptions on the minimizers of problem (\ref{poly-opt}). This result builds on  a result of Marshall \cite{Marshall2006, Marshall2009}.

Let $u$ be a local minimizer of problem (\ref{poly-opt}) and let $J(u):=\{j\in [m]: g_j(u)=0\}$ be the set of inequality constraints that are active at $u$. We say that the  \textit{constraint qualification condition} (abbreviated as CQC) holds at $u$ if the set 
$$G(u):=\{\nabla g_j(u) : j\in J(u)\} \cup \{\nabla h_i(u): i\in [l]\}$$ is linearly independent.  If CQC holds at $u$ then there exist  $\lambda_1, \dots, \lambda_l, \mu_1, \dots, \mu_m\in\oR$  satisfying
\begin{align*}
  \nabla f(u)=\sum_{i=1}^{l}\lambda_i\nabla h_i(u)+\sum_{j\in J(u)}\mu_j\nabla g_j(u), \quad \mu_j\ge 0 \text{ for } j\in J(u),\quad \\ \mu_j=0\text{ for } j\in [m]\setminus J(u).
\end{align*}
If we have $\mu_j>0$ for all $j\in J(u)$, then we say that the \textit{strict complementarity condition} (abbreviated as SCC) holds. The Lagrangian function $L(x)$ is defined as  
$$L(x):= f(x)-\sum_{i=1}^{l}\lambda_ih_i(x)-\sum_{j\in J(u)}\mu_jg_j(x).$$
Another (second order) necessary condition for $u$ to be a local minimizer is the following inequality 
\begin{align}\label{eqSONC}\tag{SONC}
    v^T\nabla^2L(u)v\geq 0 \text{ for all } v\in G(u)^{\perp}.
\end{align}
\ignore{
where $G(u)$ is defined by
$$G(u)^{\perp}=\{x\in \oR^n: x^T\nabla g_j(u)=0 \text{ for all  } j\in J(u) \text{ and }  x^T\nabla h_i(u)=0 \text{ for all } i\in [k]\}.$$}
 If it happens that the inequality (\ref{eqSONC}) is strict, i.e., if
\begin{align}\label{eqSOSC}\tag{SOSC}
    v^T\nabla^2L(u)v> 0 \text{ for all } 0\neq v\in G(u)^{\perp},
\end{align}
then one says that the \textit{second order sufficiency condition}  (SOSC) holds at $u$.

\medskip
We can now state the following result by Nie \cite{Nie}.

\begin{theorem}[Nie \cite{Nie}]\label{theo-Nie}
Assume that the Archimedean condition (\ref{archimedean-cond}) holds for the polynomial sets $g$ and $h$ in problem (\ref{poly-opt}). If the constraint qualification condition (CQC), the strict complementarity condition (SCC), and the second order sufficiency condition (SOSC)  hold at every global minimizer of (\ref{poly-opt}), then the Lasserre hierarchy (\ref{lasserre-hierarchy}) has finite convergence, i.e., $f^{(r)}=f^*$ for some $r\in \oN$.
\end{theorem}

Nie \cite{Nie} uses Theorem \ref{theo-Nie} to show that finite convergence of Lasserre hierarchy (\ref{lasserre-hierarchy}) holds generically. Note that the conditions in the above theorem imply that problem (\ref{poly-opt}) has finitely many minimizers.
So this result may help to show finite convergence only when there are finitely many minimizers.
It will be used later in this chapter (for the proof of Theorem \ref{Tpsi} and Theorem \ref{finite-acritical}).

\section{Sum-of-squares approximations for $\COP_n$}\label{secconic}
\label{sec:sos-approx}

As mentioned in the Introduction, optimizing over the copositive cone  is a hard problem,  this motivates to design tractable conic inner approximations for it.  One classical cone that is often used as inner relaxation of $\COP_n$  is the cone $\SPN_n$, defined as 
\begin{align}\label{spn}
\SPN_n:=\{M\in \MS^n: M=P+N \text{ where } P\succeq 0, N\geq0\}.
\end{align} 
In this section we explore several conic approximations for $\COP_n$, strengthening $\SPN_n$, based on sums of squares of polynomials.  They are inspired by the positivity certificates (\ref{cert-reznick}), (\ref{cert-polya}), (\ref{cert-putinar}), and (\ref{cert-schmudgen})   introduced in Section \ref{preliminaries}. 

\subsection{Cones based on P\'{o}lya's nonnegativity certificate}\label{secconeK}

In view of relation (\ref{def-COP}), a matrix is copositive if the homogeneous polynomial $x^TMx$ is nonnegative on $\mathbb{R}_+^n$. Motivated by the nonnegativity certificate (\ref{cert-polya}) in P\'{o}lya's theorem, de Klerk and Pasechnik \cite{dKP2002} introduced the cones $\MC_n^{(r)}$, defined as   
\begin{align}\label{eqCr}
\MC^{(r)}_n:=\Big\{M\in \MS^n: \Big(\sum_{i=1}^{n}x_i\Big)^{r} x^TMx  \text{ has nonnegative coefficients}\Big\}
\end{align}
for any $r\in\oN$. Clearly, $\MC_n^{(r)}\subseteq \MC^{(r+1)}_n\subseteq \COP_n$. By P\'{o}lya's theorem (Theorem \ref{theorem-polya}), the cones $\MC^{(r)}_n$ cover the interior of $\COP_n$, i.e.,  $\text{int}(\COP_n)\subseteq \bigcup_{r\geq 0}\MC_n^{(r)}$. This follows from the fact that $M\in \text{\rm int}(\COP_n)$ precisely when 
$x^TMx>0$ for all $x\in \mathbb{R}_+^{n}\setminus\{0\}$. The cones $\MC^{(r)}_n$ were introduced in \cite{dKP2002} for approximating the stability number of a graph, as we will see in Section \ref{section-alpha}. 

In a similar way, in view of relation (\ref{def-COP-sq}), a matrix is copositive if the homogeneous polynomial $(x^{\circ2})^TMx^{\circ2}$ is globally nonnegative. Parrilo \cite{Parrilo-thesis-2000} introduced the cones $\MK_n^{(r)}$, that are  defined by using certificate (\ref{cert-reznick}) as
\begin{align}\label{eqKr}
\MK^{(r)}_n:=\Big\{M\in \MS^n: \Big(\sum_{i=1}^nx_i^2\Big)^r(x^{\circ2})^TMx^{\circ2} \in\Sigma\Big\}.
\end{align}
Clearly, $\MC_n^{(r)}\subseteq \MK_n^{(r)}\subseteq \COP_n$, and thus $\text{int}(\COP_n)\subseteq \bigcup_{r\geq0}\MK_n^{(r)}$. This inclusion also follows  from Reznick's theorem (Theorem \ref{theorem-reznick}).

The following result by Pe\~{n}a, Vera and Zuluaga \cite{ZVP2006} gives information about the structure of the homogeneous polynomials $f$ for which $f(x^{\circ2})$ is a sum of squares. As a byproduct, this gives the reformulation for the cones $\MK_n^{(r)}$ from relation (\ref{eqKrb}) below.

\begin{theorem}[Pe\~{n}a, Vera, Zuluaga \cite{ZVP2006}]\label{theoZVP} 
Let $f\in \oR[x]$ be a homogeneous polynomial with degree $d$.
Then the polynomial $f(x^{\circ2})$ is a sum of squares if and only if  $f$ admits a decomposition of the  form 
  \begin{align}\label{eqf}
  f = \sum_{\substack{S\subseteq [n], |S|\leq d \\ |S|\equiv d\ (\mmod  2)} }\sigma_Sx^S \ \ \text{ for some } \sigma_S\in \Sigma_{d-|S|}.
  \end{align}
In particular, for any $r\ge 0$, we have
\begin{align}\label{eqKrb}
\MK^{(r)}_n=\Big\{M\in \MS^n: \Big(\sum_{i=1}^{n}x_i\Big)^{r} x^TMx =\sum_{\substack{S\subseteq [n], |S|\leq r+2 \\ |S|\equiv r\ (\mmod  2)} }\sigma_Sx^S \ \ \text{ for some } \sigma_S\in \Sigma_{r+2-|S|}\Big\}.
\end{align}
\end{theorem}
Alternatively, the cones $\MK^{(r)}_n$ may be defined as
\begin{align}\label{eqKrbb}
\MK^{(r)}_n=\Big\{M\in \MS^n: \Big(\sum_{i=1}^{n}x_i\Big)^{r} x^TMx =
\sum_{\substack{\beta\in \mathbb{N}^n\\ |\beta|\le r+2 }} \sigma_\beta x^\beta \ \ \text{ for some } \sigma_\beta\in \Sigma_{r+2-|\beta|}\Big\},
\end{align}
where, in (\ref{eqKrb}), one replaces square-free monomials by arbitrary monomials.
Based on this reformulation of the cones $\MK_n^{(r)}$, Pe\~{n}a et.al. \cite{ZVP2006} introduced the cones $\MQ_n^{(r)}$, 
defined as 
\begin{align}\label{eqQr}
\MQ^{(r)}_n:=\Big\{M\in \MS^n: \Big(\sum_{i=1}^{n}x_i\Big)^{r} x^TMx =
\sum_{\substack{\beta\in\oN^n \\ |\beta|=r,r+2} }\sigma_\beta x^\beta \ \ \text{ for some } \sigma_\beta\in \Sigma_{r+2-|\beta|}\Big\}.
\end{align} 
So $\MQ_n^{(r)}$ is a restrictive version of the formulation (\ref{eqKrbb}) for the cone $\MK_n^{(r)}$, in which the decomposition only allows sums of squares of degree 0 and 2. Then, we have 
\begin{align}\label{inclusion-cones}
\MC_n^{(r)}\subseteq \MQ_n^{(r)} \subseteq \MK_n^{(r)},
\end{align}
and thus 
\begin{align}\label{inclusion-union-cones}
\text{int}(\COP_n)\subseteq \bigcup_{r\geq0}\MC_n^{(r)}\subseteq \bigcup_{r\geq0}\MQ_n^{(r)}\subseteq \bigcup_{r\geq0}\MK_n^{(r)}.
\end{align}

As an application of (\ref{eqKrb}) 
we obtain the following characterization of the cones $\MK_n^{(r)}$ for $r=0,1$. A matrix $M\in \MS^n$ belongs to $\MK_n^{(0)}$ if and only if 
\begin{align*}
x^TMx=\sigma+\sum_{1\leq i<j\leq n}c_{ij}x_ix_j 
\end{align*}
for some $\sigma\in \Sigma_2$ and some scalars $c_{ij}\geq0 \text{ for }1\leq i<j\leq n$, and   $M$ belongs to $\MK_n^{(1)}$ if and only if 
\begin{align}\label{eqpM}
\Big(\sum_{i=1}^nx_i\Big)x^TMx= \sum_{i=1}^nx_i\sigma_i  + \sum_{1\leq i\leq j\leq k\leq n}c_{ijk}x_ix_jx_k,
\end{align}
for some $\sigma_i\in \Sigma_2$ for $i\in [n]$ and some scalars $c_{ijk}$ for $1\leq i\leq j\leq n$. 
From this, one can also derive the following result.

\begin{lemma}[Characterization of the cones $\MK^{(0)}_n$ and $\MK^{(1)}_n$]\label{lem-K0-K1}

 Let $M\in \MS^n$ be a symmetric matrix. Then the following holds.
 \begin{description}
\item [(1)] $M$ belongs to the cone $\MK_n^{(0)}$ if and only if there exists a positive semidefinite matrix $P\succeq0$ such that $P\leq M$. In other words, 
\begin{align}
\MK_n^{(0)}=\{M\in \MS^n: M=P+N \text{ for some } P\succeq 0 \text{ and } N\geq0\}=\SPN_n. 
\end{align}
\item [(2)] $M$ belongs to the cone $\MK^{(1)}_n$ if and only if there exist symmetric matrices $P(i)$ for $i\in [n]$ satisfying the following conditions:
\begin{description}
\item[(i)] $P(i)\succeq 0$ \  for all $i\in [n]$,
\item[(ii)] $P(i)_{ii}=M_{ii}$  \  for all $i\in [n]$,
\item[(iii)] $2P(i)_{ij}+P(j)_{ii}= 2M_{ij}+M_{ii}$  \ for all $i\ne j\in [n]$,
\item[(iv)] $P(i)_{jk}+P(j)_{ik}+P(k)_{ij} \le M_{ij}+M_{ik}+M_{jk}$ \  for all distinct $i,j,k\in [n]$.
\end{description}
\end{description}
\end{lemma}
Claim (1) and the ``if" part in (2) in the above lemma were already proved by Parrilo in \cite{Parrilo-thesis-2000}. The ``only if" part in (2) was proved by Bomze and de Klerk in \cite{Bomze}.

A matrix $P$ is called to be a $\kzeroc$ for $M$ if $P\succeq 0$ and $P\leq M$. Now we show a result that relates the zeros of the form $x^TMx$ with the kernel of its $\kzeroc$s, which will be used later in the chapter. 

\begin{lemma}[\cite{LV2021b}]\label{kernel-k0}
Let $M\in \MK_n^{(0)}$ and let $P$ be a $\kzeroc$ of $M$. If $x\in \mathbb{R}^n_+$ and $x^TMx=0$, then $Px=0$ and $P[S]=M[S]$, where $S=\{i\in[n] \text{ : } x_i>0\}$ is the support of $x$.
\end{lemma}
\begin{proof}
Since $P$ is a $\kzeroc$  there exists a matrix $N\geq 0$ such that $M=P+N$. Hence, $0=x^TMx=x^TPx +x^TNx$. Then  $x^TPx=0=x^TNx$ as $P\succeq 0$ and $N\geq 0$. This implies $Px=0$ since $P\succeq 0$. On the other hand, since $x^TNx=0$ and $N\geq 0$, we get $N_{ij}=0$ for $i,j\in S$. Hence, $M[S]=P[S]$, as $M=P+N$.
\end{proof}

\subsection{Lasserre-type approximation cones}\label{secconeLAS}

Recall the definitions (\ref{def-COP}) and (\ref{def-COP-sq}) of the copositive cone.
Clearly, in (\ref{def-COP}), 
the nonnegativity condition for $x^TMx$ can be restricted to the simplex $\Delta_n$ and, in (\ref{def-COP-sq}), the nonnegativity condition for $(x^{\circ2})^TMx^{\circ2}$ can be restricted to the unit sphere $\mathbb{S}^{n-1}$.
Based on these observations, one can now  use the positivity  certificate (\ref{cert-putinar}) or 
(\ref{cert-schmudgen}) to certify the nonnegativity on $\Delta_n$ or $\mathbb S^{n-1}$. 
This leads naturally to defining the following cones (as done in \cite{LV-COP_5}): for an integer $r\in \oN$, 
\begin{align}
\LAS_{\Delta_n}^{(r)} & :=\displaystyle \Big\{M\in \MS^n: x^TMx = \displaystyle \ \sigma_0+\sum_{i=1}^n\sigma_i x_i + q \ \text{ for    } \sigma_0\in \Sigma_r, \sigma_i\in \Sigma_{r-1},\    q\in I_{\Delta_n}        \Big\}, \label{eqLASDa} 
\end{align}
\begin{align}
\LAS_{\Delta_n, \MT}^{(r)} & =\Big \{M\in \MS^n: x^TMx  =\displaystyle
\sum_{S\subseteq [n],  |S|\le r}\sigma_Sx^S + q \
 \text{ for  } \sigma_S\in \Sigma_{r-|S|} \text{ and } q\in I_{\Delta_n}\Big\}, \label{eqLASDPa}
 \end{align}
 \begin{align}
\LAS^{(r)}_{\mathbb S^{n-1}} =\Big\{M\in \MS^n:   (x^{\circ 2})^TMx^{\circ 2}= \sigma+ q \text{ for some }  \sigma\in \Sigma_r, q\in I_{\mathbb{S}^{n-1}}\Big\}.
\end{align}
 Clearly, we have $\LAS_{\Delta_n}^{(r)} \subseteq \LAS_{\Delta_n, \MT}^{(r)}$ and,
by Putinar's theorem (Theorem \ref{theorem-putinar}),
\begin{align}\label{interior-LAS}
\text{int}(\COP_n)\subseteq \bigcup_{r\geq0} \LAS_{\Delta_n}^{(r)}, \quad \quad  \text{int}(\COP_n)\subseteq \bigcup_{r\geq0} \LAS_{\mathbb{S}^{n-1}}^{(r)}.
\end{align} 

\subsection{Links between the various approximation cones for $\COP_n$}\label{secconelinks}

In this section, we link the various cones introduced in the previous sections.

\begin{theorem}[\cite{LV-COP_5}]\label{theo-link}
 Let $r\geq 2$ and $n\geq 1$. Then the following holds.
 \begin{align}\label{link}
 \LAS_{\Delta_n}^{(r)}\subseteq\MK_n^{(r-2)}=\LAS_{\Delta_n,\MT}^{(r)}=\LAS_{\mathbb{S}^{n-1}}^{(2r)}.
 \end{align} 
 \end{theorem}
So,  this result shows  that membership in the cones $\MK^{(r)}_n$ can be  characterized via positivity certificates on $\oR^n_+$ or $\oR^n$ of P\'olya- and Reznick-type (using a 'denominator' of the form $(\sum_ix_i)^r$ for some $r\in\oN$), or, alternatively, via `denominator-free' positivity certificates on the simplex or the sphere of Schm\"udgen- and Putinar-type.

Theorem \ref{theo-link}   was implicitly shown in \cite[Corollary 3.9]{LV2021a}. We now sketch the proof. First, the equality $\MK_n^{(r-2)}=\LAS_{\mathbb{S}^{n-1}}^{(2r)}$ follows from the following result.
\begin{theorem}[de Klerk, Laurent, Parrilo \cite{dKLP}]\label{dKPL-prop}
Let $f$ be a homogeneous polynomial of  degree $2d$ and $r\in \oN$. Then, we have $(\sum_{i=1}^nx_i^2)^rf\in \Sigma $ if and only if $f=\sigma+u(\sum_{i=1}^nx_i^2-1)$ for some $\sigma\in \Sigma_{2r+2d}$ and  $u\in \oR[x]$.

In particular, for any $r\ge 2$, we have 
\begin{align}
 \LAS^{(2r)}_{\mathbb S^{n-1}} =\Big\{M\in \MS^n:  \Big(\sum_{i=1}^n x_i^2\Big)^{r-2} (x^{\circ 2})^TMx^{\circ 2} \in \Sigma\Big\}=\MK^{(r-2)}_n.\label{eqLASSb}
 \end{align}
  \end{theorem}

Next, the  inclusion $\LAS_{\Delta_n,\MT}^{(r)}\subseteq \LAS_{\mathbb{S}^{n-1}}^{(2r)}$ follows by replacing $x$ by $x^{\circ2}$ in the definition of $\LAS_{\Delta_n,\MT}^{(r)}$. Indeed, if $M\in \LAS_{\Delta_n,\MT}^{(r)}$, then 
$$x^TMx  = \sum_{S\subseteq [n],  |S|\le r}\sigma_Sx^S + q\Big(\sum_{i=1}^nx_i-1\Big) \text{ for  } \sigma_S\in \Sigma_{|S|-r}, q\in \mathbb{R}[x].$$
Then, by replacing $x$ by $x^{\circ2}$, we obtain 
$$(x^{\circ2})^TMx^{\circ2}  = \sum_{\substack{S\subseteq [n]\\ |S|\le r}}\sigma_S(x^{\circ2})\prod_{i\in S}x_i^2 + q(x^{\circ2})\Big(\sum_{i=1}^nx_i^2-1\Big) \text{ for  } \sigma_S\in \Sigma_{|S|-r}, q\in \mathbb{R}[x],$$
where the first summation is a sum of squares of degree at most $2r$, thus showing 
 that $M\in  \LAS_{\mathbb{S}^{n-1}}^{(2r)}$. 

Finally, as the inclusion $\LAS^{(r)}_{\Delta_n}\subseteq \LAS^{(r)}_{\Delta_n,\MT}$ is clear, 
it remains to show that $\MK_n^{(r-2)}\subseteq \LAS_{\Delta_n,\MT}^{(r)}$ in order to conclude the proof of Theorem \ref{theo-link}. For this, we use the formulation (\ref{eqKrb}) of the cones $\MK_n^{(r)}$. Let $M\in \MK_n^{(r-2)}$, then 
$$ \Big(\sum_{i=1}^{n}x_i\Big)^{r-2} x^TMx =\sum_{\substack{S\subseteq [n], |S|\leq r \\ |S|\equiv r\ (\mmod  2)} }\sigma_Sx^S \ \ \text{ for some } \sigma_S\in \Sigma_{r-|S|}.$$
Write  $\sum_{i=1}^nx_i= (\sum_{i=1}^nx_i-1) +1$ and expand 
$(\sum_{i=1}^nx_i)^r$ as $1 + p(\sum_{i=1}^nx_i-1)$ for some $p\in \oR[x]$.
From this, setting $q=-p x^TMx$, we  obtain 
$$x^TMx =\sum_{\substack{S\subseteq [n], |S|\leq r \\ |S|\equiv r\ (\mmod  2)} }\sigma_Sx^S + q\Big(\sum_{i=1}^n x_i-1\Big) \ \ \text{ for some } \sigma_S\in \Sigma_{r+2-|S|}, \ q\in \oR[x],$$
which shows $M\in  \LAS_{\Delta_n,\MT}^{(r)}$.

 It is useful  to note that, in the formulation (\ref{eqLASDPa}) of $\LAS^{(r)}_{\Delta_n,\MT}$,
 we could equivalently require a decomposition of the form
 \begin{equation}\label{eqbeta}
 x^TMx=\sum_{\beta\in \oN^n, |\beta|\le r}\sigma_\beta x^\beta +q\ \text{ for some } \sigma_\beta\in \Sigma_{r-|\beta|}\ \text{ and } q\in I_{\Delta_n},
 \end{equation}
 thus using arbitrary monomials $x^{\beta}$ instead of square-free monomials $x^S$.
This allows to draw a parallel with the definitions of the cones $\MC_n^{(r)}$ (in (\ref{eqCr})) and $\MQ_n^{(r)}$ (in (\ref{eqQr})). Namely,  using the same type of arguments as above,  one can obtain the following analogous  reformulations for the cones $\MC_n^{(r)}$ and $\MQ_n^{(r)}$:
\begin{align}
\MQ_n^{(r)}=\Big\{M\in \MS^n: x^TMx =
\sum_{\substack{\beta\in \mathbb{N}^n\\ |\beta|=r, r+2 }} \sigma_\beta x^\beta + q \ \text{ for  } \sigma_\beta\in \Sigma_{r+2-|\beta|} \text{ and } q\in I_{\Delta_n}\Big\}, \label{eqQrb}
\end{align}
\begin{align}
\MC_n^{(r)}=\{M\in \MS^n: x^TMx= \sum_{\substack{\beta\in \mathbb{N}^n\\|\beta|=r+2}}c_\beta x^\beta+ q \ \text{ for } c_\beta\geq0 \text{ and }q\in I_{\Delta_n}\}.
\end{align}

\begin{example}{Seeing all cones as restrictive Schm\"udgen-type representations of $x^TMx$}
 We illustrate how membership in the cones $\LAS^{(r)}_{\Delta_n}$, $\LAS^{(r)}_{\Delta_n,\MT}$, $\MC^{(r)}_n$, and $\MQ^{(r)}_n$ can also be viewed as `restrictive' versions of membership in the cone $\MK^{(r-2)}_n$. Indeed, as we saw above, $\MK^{(r-2)}_n=\LAS^{(r)}_{\Delta_n,\MT}$ and thus  a matrix $M$ belongs to $\MK^{(r-2)}_n$ if and only if the form $x^TMx$ has a decomposition of the form (\ref{eqbeta}). 
 Then, membership in the cones $\LAS^{(r)}_{\Delta_n}$, $\MC^{(r-2)}_n$, and $\MQ^{(r-2)}_n$ corresponds to restricting  to decompositions that allow only some terms in (\ref{eqbeta}):
\begin{align}\label{link-rep}
\underbrace{\sigma_0 + \sum_{i=1}^nx_i\sigma_i}_{\text{for cones $\LAS_{\Delta_n}^{(r)}$}} + \dots +\overbrace{\sum_{\beta \in \mathbb{N}^n, |\beta|=r-2}x^\beta \sigma_\beta +  \underbrace{\sum_{\beta \in \mathbb{N}^n, |\beta|=r}x^\beta c_\beta}_{\text{for cones $\MC_n^{(r-2)}$}}}^{\text{for cones $\MQ_n^{(r-2)}$}} + \underbrace{q(\sum_{i=1}^nx_i-1)}_{\text{for cones }\begin{cases}\LAS_{\Delta_n}^{(r)}\\ \MQ_n^{(r-2)} \\ \MC_n^{(r-2)} \end{cases}}
\end{align}
\end{example}

\section{Exactness of sum-of-squares approximations for $\COP_n$}\label{secexact}

We have discussed  several hierarchies of conic inner approximations for the copositive cone $\COP_n$. In particular, we have seen that each of them  covers 
the interior of $\COP_n$. 
In this section, we investigate the question of deciding exactness of these hierarchies, where we say that a hierarchy of conic inner approximations is {\em exact} if it covers the full copositive cone $\COP_n$.

\subsection{Exactness of the conic approximations $\MK_n^{(r)}$}\label{secexactK}

We first recall a result from \cite{Diananda}, that shows  equality  in the inclusion $\MK_n^{(0)}\subseteq \COP_n$ for $n\leq 4$.
\begin{theorem}[Diananda \cite{Diananda}]\label{theo-k0}
For $n\leq 4$ we have 
$$\COP_n=\{M\in\MS^n:M=P+N \text{ for some } P\succeq 0, N\geq 0\}=\MK_n^{(0)}\ (=\SPN_n).$$
\end{theorem}
This result does not extend to matrix size $n\geq 5$.  For instance, as we now see,  the {\em Horn matrix} $H$ in (\ref{matHorn})  is copositive, but it does not belong to $\MK_5^{(0)}$.

\begin{example}{The Horn matrix}
The Horn matrix 
\begin{align}\label{matHorn}
H := \left(\begin{matrix} 1 & 1 & -1 & -1 & 1\cr
1 & 1 & 1 & -1 & -1\cr
-1 & 1 & 1 & 1 & -1\cr
-1 & -1 & 1 & 1 & 1 \cr
1 & -1 & -1 & 1 & 1
 \end{matrix}\right)
\end{align}
is copositive. A direct way to show this  is to observe that  $H\in \MK_n^{(1)}$.  Parrilo \cite{Parrilo-thesis-2000} shows this latter fact by giving the following explicit sum of squares decomposition:
\begin{equation}
\begin{split}\label{Horn-k1}
\Big(\sum_{i=1}^5x_i^2\Big)(x^{\circ2})^THx^{\circ2}= &\hspace{0.2cm}x_1^2(x_1^2+x_2^2+x_5^2-x_3^2-x_4^2)^2 \\
&+x_2^2(x_1^2+x_2^2+x_3^2-x_4^2-x_5^2)^2\\
&+x_3^2(x_2^2+ x_3^2+x_4^2-x_5^2-x_1^2)^2\\
&+x_4^2(x_3^2+x_4^2+x_5^2-x_1^2-x_2^2)^2\\
&+x_5^2(x_1^2+x_4^2+x_5^2-x_2^2-x_3^2)^2\\
&+4x_1^2x_2^2x_5^2+ 4x_1^2x_2^2x_3^2\\
&+ 4x_2^2x_3^2x_4^2+ 4x_3^2x_4^2x_5^2\\
&+ 4x_4^2x_5^2x_1^2.
\end{split}
\end{equation}
\end{example}
On the other hand, Hall and Newman \cite{Hall-Newman} show that $H$ does not belong to $\SPN_5$ ($=\MK_5^{(0)}$). We give a short proof of this fact, based on Lemma \ref{kernel-k0}.

\begin{theorem}[Hall, Newman \cite{Hall-Newman}]\label{H-notin-K_0}
The Horn matrix $H$ does not belong to $\MK^{(0)}_5$. Hence, the inclusion $\MK_n^{(0)}\subseteq \COP_n$ is strict for any $n\geq 5$.
\end{theorem}
\begin{proof}
Assume, by way of contradiction, that $H\in \MK_5^{(0)}$. Let $P$ be a $\kzeroc$ for $H$, i.e., such that $P\succeq 0$ and $P\le H$, and let $C_1, C_2, \dots, C_5$ denote the columns of $P$. Observe that $u_1=(1, 0, 1, 0 , 0)$ and $u_2=(1,0,0,1,0)$ are zeros of the form $x^THx$. Then, by Lemma \ref{kernel-k0}, $Pu_1=Pu_2=0$.
Hence, $C_1+C_3=C_1+C_4=0$, so that $C_3=C_4$. Using an analogous argument we obtain that $C_1=C_2=\ldots=C_5$, which implies  $P=tJ$ for some scalar $t\geq 0$, where $J$ is the all-ones matrix. This leads to a  contradiction since $P\leq H$.  
\end{proof}
Next, we recall a result of Dickinson, D\"ur, Gijben and Hildebrand \cite{DDGH} that shows exactness of the conic approximation $\MK_5^{(1)}$ for copositive matrices with an all-ones diagonal.

\begin{theorem}[Dickinson, D\"ur, Gijben, Hildebrand \cite{DDGH}]\label{theoDDGH1}
Let $M\in \COP_5$ with $M_{ii}=1$ for all $i\in[5]$. Then $M\in \MK_5^{(1)}$.
\end{theorem}
In contrast, the same authors show that the cone $\COP_n$ is never equal to a single cone $\MK_n^{(r)}$ for $n\geq 5$. 

\begin{theorem}[Dickinson, D\"ur, Gijben, Hildebrand \cite{DDGH}]\label{theoDDGH}
For any $n\geq 5$ and $r\geq 0$, we have $\COP_n\neq \MK_n^{(r)}$.
\end{theorem}

\begin{proof}
Let $M$ be a copositive matrix that lies outside $\MK_n^{(0)}$.  Clearly, any positive diagonal scaling of $M$ remains copositive, that is, $DMD\in \COP_n$ for any $D\in \MD^n_{++}$. We will show that for any $r\ge 0$ there exists a diagonal matrix $D\in\MD^n_{++}$ such that $DMD\not\in \MK^{(r)}_n$.
Fix $r\ge 0$ and assume, by way of contradiction, that $DMD\in \MK_n^{(r)}$ for any positive diagonal matrix $D$. Then, for all scalars $d_1, d_2, \dots, d_n>0$ the polynomial 
$(\sum_{i=1}^nx_i^2)^r(\sum_{i,j=1}^n M_{ij}d_id_jx_i^2x_j^2)$ is a sum of squares. 
Equivalently, the polynomial $(\sum_{i=1}^nd_i^{-1}z_i^2)^r(\sum_{i,j=1}^nM_{ij}z_i^2z_j^2)$ is a sum of squares in the variables $z_i=\sqrt{d_i}x_i$ ($i=1, \dots, n$). Now we fix $d_1=1$ and we let $d_i\to \infty$ for $i=2, \dots, n$. Since the cone of sums of squares of polynomials is closed (see, e.g.,   \cite[Section~3.8]{monique-survey}), the limit polynomial $(z_1^2)^r(\sum_{i,j=1}^nM_{i,j}z_i^2z_j^2)$ is also a sum of squares in the variables $z_1, \dots, z_n$. Say $(z_1^2)^r(\sum_{i,j=1}^nM_{i,j}z_i^2z_j^2)=\sum_{k=1}^mq_k^2$. 
Then, for each $k$, we have $q_k(z)=0$ whenever $z_1=0$. Hence, if $r\ge 1$, then $z_1$ can be factored out from $q_k$, and we obtain  that $(z_1^2)^{r-1}(\sum_{i,j=1}^nM_{i,j}z_i^2z_j^2)$ is also a sum of squares. After repeatedly using this argument we can conclude that $\sum_{i,j=1}^nM_{i,j}z_i^2z_j^2$ is a sum of squares, that is, $M\in \MK_n^{(0)}$, leading to a contradiction.
\end{proof}

As was recalled earlier, sums of squares of polynomials can be expressed using semidefinite programming. Hence, the cone $\MK^{(r)}_n$ is {\em semidefinite representable}, which means that membership in  it can be modeled using semidefinite programming. 
In \cite{BKT} it is  shown  that $\COP_5$ is {\em not} semidefinite representable, which is thus a stronger result that implies Theorem \ref{theoDDGH}.
On the other hand, it was  shown recently in \cite{SV} that 
every $5\times 5$ copositive matrix belongs to the cone $\MK^{(r)}_5$ for some $r\in \oN$.

\begin{theorem}[Laurent, Vargas \cite{LV-COP_5}; Schweighofer, Vargas \cite{SV}]\label{cop5=k}
We have $\COP_5=\bigcup_{r\geq 0}\MK_5^{(r)}$.
\end{theorem}
We will return to this result in Section \ref{secCOP5}, where we will give some hints on the strategy and  tools that are used for the proof.

\medskip
It is known that the result from Theorem \ref{cop5=k}  does not extend to matrix size $n\geq 6$. To show this, we recall the following  result. 

\begin{proposition}[\cite{LV2021b}]\label{theo-cop}
Let $M_1\in \COP_n$ and $M_2\in \COP_m$ be two copositive matrices. Assume $M_1\notin \MK_n^{(0)}$ and there exists $0\neq z\in \mathbb{R}^m_+$ such that $z^TM_2z=0$. Then  we have
\begin{equation}
\left(
\begin{array}{c|c}
M_1 & 0\\
\hline
0 & M_2
\end{array}
\right) \in \COP_{n+m} \setminus \bigcup\limits_{r\in \mathbb{N}}\MK_{n+m}^{(r)}.
\end{equation}
\end{proposition} 
Now we give explicit examples of copositive matrices of size $n\ge 6$ that do not belong to any of the cones $\MK_n^{(r)}$.

\begin{example}{Examples of copositive matrices outside $\bigcup_{r\geq0}\MK_n^{(r)}$}
Let  $M_1=H$ be the Horn matrix, known   to be copositive with $H\notin \MK_n^{(0)}$. For the matrix $M_2$ we first consider the $1\times 1$ matrix $M_2=0$ and, as a second example, we consider $ M_2=\begin{pmatrix} 1 & -1 \\ -1 & 1 \end{pmatrix}\in \COP_2$. Then,  as an application of Proposition~\ref{theo-cop},  we obtain 
\begin{equation}\label{eqexmat}
\left(
\begin{array}{c|c}
H & 0\\
\hline
0 & 0
\end{array}
\right) \in \COP_6 \setminus \bigcup\limits_{r\in\mathbb{N}}\MK_{6}^{(r)}, \quad  \left(
\begin{array}{c|c}
H & 0\\
\hline
0 & {\begin{array}{cc}1 & -1 \\ -1 & 1
\end{array}}
\end{array}
\right)\in \COP_7\setminus \bigcup\limits_{r\in \mathbb{N}}\MK_7^{(r)}.
 \end{equation}
The leftmost matrix in (\ref{eqexmat}) is  copositive, it has all its diagonal entries equal to $0$ or $1$, and it does not belong to any of the cones $\MK^{(r)}_6$. Selecting for $M_2$ the zero matrix of  size $m\ge 1$  gives  a matrix in $\COP_n\setminus \bigcup_{r\ge 0}\MK^{(r)}_n$ for any size $n\ge 6$.
The rightmost matrix in (\ref{eqexmat}) is  copositive,  it has all its diagonal entries equal to 1, and  it does not lie in any of the cones $\MK_7^{(r)}$.
More generally, if we  select the matrix $M_2= {1\over m-1}(mI_m-J_m)$, which is positive semidefinite with $e^TM_2e=0$, then we obtain a matrix in $\COP_n\setminus \bigcup_{r\ge 0} \MK^{(r)}_n$ with an all-ones diagonal for any size $n\ge 7$.
In contrast, as mentioned in Theorem \ref{theoDDGH1}, 
any copositive $5\times 5$ matrix with an all-ones diagonal  belongs to $\MK_5^{(1)}$. 
The situation for the case of $6\times 6$ copositive matrices remains open.
\end{example}

\begin{example}{Question}
Is it true that any $6\times 6$ copositive matrix with an all-ones diagonal belongs to $\MK^{(r)}_6$ for some $r\in \oN$?
\end{example}

\subsection{Exactness of the conic approximations $\LAS^{(r)}_{\Delta_n}$}
\label{section-membership-LAS}
We begin with the characterization of the matrix sizes $n$ for which the hierarchy of cones $\LAS_{\Delta_n}^{(r)}$ is exact.

\begin{theorem}[Laurent, Vargas \cite{LV-COP_5}]\label{theoexactLAS}
We have $\COP_2=\LAS_{\Delta_2}^{(3)}$, and the inclusion $\bigcup_{r\ge 0}\LAS_{\Delta_n}^{(r)}\subseteq \COP_n$ is strict for any $n\ge 3$.
\end{theorem}
\begin{proof}
First, assume $M=\left(\begin{matrix}a & c \cr c & b\end{matrix}\right)\in \COP_2$, we show $M\in \LAS^{(3)}_{\Delta_2}$. Note that $a,b\ge 0$ and $c\ge -\sqrt{ab}$ (using the fact that $u^TMu\ge 0$ with $u=(1,0),$ $ (0,1),$ and $(\sqrt b,\sqrt a)$). Then we can write
$x^TMx= (\sqrt ax_1 -\sqrt b x_2)^2 + 2(c+\sqrt{ab})x_1x_2$, which, modulo the ideal $I_{\Delta_2}$,  is equal to 
$ (\sqrt ax_1 -\sqrt b x_2)^2 (x_1+x_2)+ 2(c+\sqrt{ab})(x_2^2x_1+x_1^2x_2)$, thus showing   
$M\in \LAS^{(3)}_{\Delta_2}$.

For $n=3$, the  matrix
\begin{align}\label{matrix-M}
 M:= \begin{pmatrix} 0 & 1 & 0 \\ 1 & 0 & 0 \\ 0 & 0 & 0 \end{pmatrix}
\end{align}
is copositive (since nonnegative), but does not belong to any of the cones $\LAS_{\Delta_3}^{(r)}$. 
To see this, assume, by way of contradiction, that $M\in\LAS^{(r)}_{\Delta_3}$ for some $r\in\oN$. Then the polynomial $x^TMx =2x_1x_2$ has a decomposition as in (\ref{no-finite}). However, we showed in the related example (end of Section \ref{secpop}) that such a decomposition does not exist.
\end{proof}

\begin{example}{Some differences between the cones $\LAS_{\Delta_n}^{(r)}$ and $\MK_n^{(r)}$}
By Theorems \ref{theo-link} and \ref{theoexactLAS}, we have $\bigcup_r \LAS^{(r)}_{\Delta_n} \subseteq \bigcup_r\MK^{(r)}_n$, with equality if $n=2$. This inclusion is strict for any $n\ge 3$. Indeed, the 
 matrix $M$ in (\ref{matrix-M}) is an example of a matrix that does not belong to any cone $\LAS_{\Delta_3}^{(r)}$ while it belongs to the cone $\MK_3^{(0)}$ (because $M$  is copositive and $\COP_3=\MK_3^{(0)}$, in view of Theorem \ref{theo-k0}). 

Another example is the Horn Matrix $H$. As observed in (\ref{Horn-k1}), $H\in \MK_5^{(1)}$ and it can be shown that $H\notin \LAS_{\Delta_5}^{(r)}$ for any $r$ (see \cite{LV-COP_5}). The proof  exploits the structure of the (infinitely many) zeros of the form $x^THx$ in $\Delta_5$.
\end{example}

We just saw two examples of copositive matrices that do not belong to any cone $\LAS_{\Delta_n}^{(r)}$. In both cases, the structure of the {\em infinitely many} zeros plays a crucial role. 
We will now discuss some tools that can be used to show membership in some cone $\LAS^{(r)}_{\Delta_n}$ 
in the case when the quadratic form $x^TMx$ has {\em finitely many} zeros in  $\Delta_n$.

\medskip
First, recall that, if a matrix $M$ lies in the interior of the cone $\COP_n$, then it belongs to some cone $\LAS^{(r)}_{\Delta_n}$ (see relation (\ref{interior-LAS})). Therefore we  now assume that $M$ lies on the boundary of $\COP_n$, denoted by $\partial \COP_n$.
The next result shows that, if the quadratic form $x^TMx$ has finitely many zeros in $\Delta_n$ and if these zeros satisfy an additional technical condition, then $M$ belongs to some cone $\LAS_{\Delta_n}^{(r)}$.   

\begin{theorem}[Laurent, Vargas \cite{LV-COP_5}]\label{opt-DMD}
Let $M\in \bCOP_n$. Assume that the quadratic form $p_M:= x^TMx$ has finitely many zeros in $\Delta_n$ and that, for every zero $u$ of $p_M$ in $\Delta_n$, we have $(Mu)_i>0$ for all $i\in [n]\setminus \supp(u)$.
Then, $M\in \bigcup_{r\geq0}\LAS_{\Delta_n}^{(r)}$ and, moreover, $DMD\in \bigcup_{r\geq0}\LAS_{\Delta_n}^{(r)}$ for all $D\in \MD_{++}^n$.
\end{theorem}

The proof of Theorem \ref{opt-DMD} relies on  following an optimization approach, which enables using the result from Theorem \ref{theo-Nie} about finite convergence of the Lasserre hierarchy. 
For this, consider the following standard quadratic program 
\begin{align}\label{sqp}
\min\{x^TMx: x\in \Delta_n\}.
\end{align} 
First, since $M\in \partial\COP_n$ the optimal value of problem (\ref{sqp}) is zero and thus a vector $u\in \Delta_n$ is a global minimizer of problem (\ref{sqp}) if and only if $u$ is a zero of $x^TMx$.
Next, observe that, as a direct consequence of the definitions,  showing membership in some cone $\LAS^{(r)}_{\Delta_n}$ amounts to showing finite convergence of the Lasserre hierarchy for problem (\ref{sqp}).
\ignore{
In this section we will study sufficient conditions for the copositive matrix $M$ for belonging to the union $\bigcup_{r\geq 0}\LAS_{\Delta_n}^{(r)}$. Consider the following standard quadratic program 
\begin{align}\label{sqp}
p_M^*=\min\{x^TMx: x\in \Delta_n\}.
\end{align} 
Since $M$ is copositive we have $p^*\geq 0$. If $p^*>0$, then $M\in \text{int}(\COP_n)$ and thus, in view of (\ref{interior-LAS}), $M\in \bigcup_{r\geq0}\LAS_{\Delta_n}^{(r)}$. Hence, we restrict our attention to the cases where $p_M^*=0$, i.e., $M\in \partial\COP_n$. With the assumption $p^*_M=0$, it is easy to notice that, $M\in \bigcup_{r\geq 0}\LAS_{\Delta_n}^{(r)}$ if and only if the Lasserre hierarchy (\ref{lasserre-hierarchy}) for problem (\ref{sqp}) has finite convergence. 
}
\begin{example}{ Linking membership in  $\LAS^{(r)}_{\Delta_n}$ to finite convergence of Lasserre hierarchy}
Assume $M\in \partial\COP_n$. Then, $M\in \bigcup_{r\geq0}\LAS_{\Delta_n}^{(r)}$ if and only if the Lasserre hierarchy (\ref{lasserre-hierarchy}) applied to  problem (\ref{sqp}) (for matrix $M$)  has finite convergence.
\end{example}

Now, in order to study the finite convergence of the Lasserre hierarchy for problem (\ref{sqp}), we will apply the result of Theorem \ref{theo-Nie} to the special case of problem~(\ref{sqp}).
First, we observe that the Archimedean condition holds. For this, note that, for any $i\in [n]$, we have  
$$1-x_i= 1-\sum_{k=1}^n x_k +\sum_{k\in[n]\setminus \{i\}}x_k, \quad 1-x_i^2 = {(1+x_i)^2\over 2}(1-x_i) + {(1-x_i)^2\over 2}(1+x_i).$$
This implies $n-\sum_{i=1}^n x_i^2 \in \mathcal{M}(x_1,\ldots,x_n) +I_{\Delta_n}$, 
thus showing that the Archimedean condition holds. 

In \cite{LV-COP_5}  it is shown that the strict complementarity condition (SCC) holds at a global minimizer 
$u$ of problem (\ref{sqp}) if and only if $(Mu)_i>0$ for all $i\in [n]\setminus \supp(u)$. 
It is also shown there that, if problem (\ref{sqp}) has finitely many minimizers, then the second order sufficiency condition (SOSC) holds at each of them.  These two facts (roughly) allow us to apply the result from Theorem \ref{theo-Nie} and to conclude the proof of Theorem \ref{opt-DMD}.
The exact technical details are summarized in the next result.

\begin{proposition}[\cite{LV-COP_5}]\label{bla}
Let $M\in \partial\COP_n$ and $D\in \MD^n_{++}$.  Assume the form $x^TMx$ has finitely many zeros in $\Delta_n$. Then the following holds.
\begin{enumerate}
\item[(i)] (SCC) holds at a  minimizer $u$ of problem (\ref{sqp}) (for $M$) 
if $(Mu)_i>0$ for all $i\in [n]\setminus \supp(u)$.
\item[(ii)] (SOSC) holds at every minimizer of problem (\ref{sqp}) (for $M$).
\end{enumerate}
In addition, if the optimality conditions (SCC) and (SOSC) hold at every minimizer of problem (\ref{sqp}) for the matrix $M$, then they also hold  for every minimizer of problem (\ref{sqp}) for the matrix $DMD$.
\end{proposition}

The following example shows a copositive matrix $M$ for which the form $x^TMx$ 
has a unique zero in $\Delta_n$; however $M$ does not belong to  $\bigcup_{r\geq0}\MK_n^{(r)}$, and thus it also does not belong to  $\bigcup_{r\geq0}\LAS_{\Delta_n}^{(r)}$ (in view of relation (\ref{link})).
Hence, the condition on the support of the zeros in Theorem \ref{opt-DMD} cannot be omitted.

\begin{example}{A copositive matrix with a unique zero, that does not belong to any cone $\MK^{(r)}_n$}
Let $M_1$ be a matrix lying in $\text{int}(\COP_n)\setminus \MK^{(0)}_n$.
Such a matrix exists for any $n\geq 5$.
As an example for $M_1$, one may take the Horn matrix $H$ in (\ref{matHorn}), in which we replace all entries  1 by $t$, where  $t$ is a given scalar such that $1<t<\sqrt 5-1$ (see  \cite{LV2021b}).
By Theorem \ref{theo-cop} we have 
\begin{equation}
M:=\left(
\begin{array}{c|c}
M_1 & 0\\
\hline
0 & \begin{array}{cc}
1 & -1 \\
-1 & 1
\end{array}
\end{array}
\right) \in \COP_{n+2}\setminus \bigcup_{r\ge 0}\MK_{n+2}^{(r)}.
\end{equation}
Now we prove that the quadratic form $x^TMx$ has a unique zero in the simplex. For this, let $x\in \Delta_{n+2}$ such that $x^TMx=0$. As $M_1$ is strictly copositive and $y:=(x_1,\ldots,x_n)$ is a zero of the quadratic form $y^TM_1y$ it follows that $x_1=\ldots=x_n=0$. Hence  $(x_{n+1}, x_{n+2})$ is  a zero of the quadratic form $x_{n+1}^2 - 2x_{n+1}x_{n+2}+x_{n+2}^2$ in the simplex $\Delta_2$ and thus $x_{n+1}=x_{n+2}=1/2$. This shows that  the only zero of the quadratic form $x^TMx$ in the simplex $\Delta_n$ is $x=(0,0,\dots, 0, \frac{1}{2},\frac{1}{2})$, as desired.
\end{example}

\subsection{The cone of $5\times 5$ copositive matrices}\label{secCOP5}

 In this section we return to the cone $\COP_5$, more specifically, to  
 the result in Theorem~\ref{cop5=k}  claiming that $\COP_5=\bigcup_r \MK^{(r)}_5$.
 Here we give a sketch of proof for (some of) the main arguments that are used to show this result.
 
 As a starting point, observe that it suffices to show that every $5\times 5$ copositive matrix that lies on an extreme ray of $\COP_5$ (for short, call such a matrix  {\em extreme}) belongs to some cone $\MK^{(r)}_5$.
Then, as a crucial ingredient, we use the fact that the extreme matrices in $\COP_5$ have been fully characterized by Hildebrand \cite{Hildebrand}.
Note that, if $M$ is an extreme matrix in $\COP_n$, then the same holds for all its positive diagonal scalings $DMD$ where $D\in\MD^n_{++}$.
Hildebrand \cite{Hildebrand} introduced  the following matrices 
\begin{align*}
T(\psi)=
\begin{pmatrix}
1 & -\cos \psi_4 & \cos(\psi_4+\psi_5) & \cos(\psi_2+\psi_3) & -\cos\psi_3 \\
-\cos \psi_4 & 1 & -\cos\psi_5 & \cos(\psi_5+\psi_1) & \cos(\psi_3+\psi_4)  \\
\cos(\psi_4+\psi_5) & -\cos\psi_5 & 1 & -\cos\psi_1 & \cos(\psi_1+\psi_2)  \\
\cos(\psi_2+\psi_3) & \cos(\psi_5+\psi_1) & -\cos\psi_1 & 1 & -\cos\psi_2  \\
-\cos\psi_3 & \cos(\psi_3+\psi_4) & \cos(\psi_1+\psi_2)  & -\cos\psi_2 & 1  \\
\end{pmatrix},
\end{align*}
where $\psi\in \oR^5$, which he used to  
prove the following theorem.
 
 \begin{theorem}[Hildebrand \cite{Hildebrand}]\label{extreme-rays}
The extreme matrices $M$ in  $\COP_5$ can be divided into the following three categories:
\begin{description}
\item [(i)] $M\in \MK_n^{(0)}$,
\item [(ii)] $M$ is (up to row/column permutation)  a positive diagonal scaling of the Horn matrix $H$,
\item [(iii)]   $M$ is (up to row/column permutation)  a positive diagonal scaling  of a matrix $T(\psi)$ for some $\psi \in \Psi$, where the set $\Psi$ is defined by
\begin{align}\label{psi}
\Psi =\Big\{\psi\in \mathbb{R}^5 :  \sum_{i=1}^5 \psi_i < \pi, \ \psi_i>0 \text{ for } i\in [5]\Big\}.
\end{align}
\end{description}
\end{theorem}
As a direct consequence, in order to show equality $\COP_5=\bigcup_{r\geq0}\MK_n^{(r)}$,
it suffices  to show that every positive diagonal scaling of the matrices  $T(\psi)$ ($\psi\in\Psi$) and $H$ 
lies in some cone $\MK^{(r)}_n$.
It turns out that a different proof strategy is needed for the class of matrices $T(\psi)$ and for the Horn matrix $H$. The main reason lies in the fact that the form $x^TMx$ has finitely many zeros in the simplex when $M=T(\psi)$, but infinitely many zeros when $M=H$. We will next discuss these two cases separately.

\subsubsection*{Proof strategy for the matrices $T(\psi)$}

Here we show that any positive diagonal scaling of a matrix $T(\psi)$
(with $\psi\in \Psi$) belongs to some cone $\MK^{(r)}_5$.
We, in fact, show a stronger result, namely membership in some cone $\LAS^{(r)}_{\Delta_n}$.
For this,
the strategy is  to apply the result of Theorem \ref{opt-DMD} to the matrix $T(\psi)$.
So we need to verify that the required conditions on the zeros of $x^T T(\psi)x$ are satisfied. 
First, we recall a characterization of the (finitely many) zeros of $x^T T(\psi)x$, which follows from results in \cite{Hildebrand}.
\begin{lemma}[\cite{Hildebrand}]\label{minimizers}
For any $\psi\in \Psi$, the zeros of the quadratic form $x^T T(\psi) x$ in the simplex $\Delta_5$ 
are the vectors $v_i=\frac{u_i}{\|u_i\|_1}$ for $i\in [5]$, where the $u_i$'s are defined by
\begin{equation*}
\resizebox{0.9\hsize}{!}{$ u_1=\begin{pmatrix}  \sin \psi_5 \\
\sin(\psi_4+\psi_5)\\
\sin \psi_4\\
0\\
0
\end{pmatrix}, \
u_2= 
\begin{pmatrix}
\sin(\psi_3+\psi_4)\\
\sin \psi_3\\
0\\
0\\
\sin \psi_4
\end{pmatrix}, \ 
u_3= \begin{pmatrix}
0\\
\sin \psi_1\\
\sin(\psi_1+\psi_5)\\
\sin\psi_5\\
0
\end{pmatrix},\
u_4=\begin{pmatrix}
0\\
0\\
\sin\psi_2\\
\sin(\psi_1+\psi_2)\\
\sin\psi_1
\end{pmatrix},\
u_5=\begin{pmatrix}
\sin\psi_2\\
0\\
0\\
\sin\psi_3\\
\sin(\psi_2+\psi_3)
\end{pmatrix}$.}
\end{equation*}
\end{lemma}
Then, it is  straightforward to check that the conditions in Theorem \ref{opt-DMD} are satisfied and so we obtain the following result for the extreme matrices of type (iii) in Theorem~\ref{extreme-rays}.

\begin{theorem}[Laurent, Vargas \cite{LV-COP_5}]\label{Tpsi}
We have $DT(\psi)D\in \bigcup_{r\geq0}\LAS_{\Delta_n}^{(r)}$ for all $D\in \MD^5_{++}$ and $\psi\in \Psi$.
\end{theorem}

\subsubsection*{Proof strategy for the Horn matrix $H$}

As already mentioned, the above strategy cannot be applied to the positive diagonal scalings of $H$ (extreme matrices of type (ii)  in Theorem \ref{extreme-rays}), because the form $x^THx$ has infinitely many zeros in $\Delta_5$; e.g.,  any $x=(\frac{1}{2}, 0, \frac{t}{2}, \frac{1-t}{2}, 0)$ with $t\in [0,1]$ is a zero. In fact, as mentioned earlier, the Horn matrix $H$ does not belong to any of the cones $\LAS_{\Delta_n}^{(r)}$ (see \cite{LV-COP_5}).
 Then, another strategy should be applied  for showing that all its positive diagonal scalings belong to some cone $\MK^{(r)}_5$.

\medskip
The starting point is to use  the fact that $\bigcup_r \MK^{(r)}_n=\bigcup _r \LAS^{(r)}_{\mathbb S^{n-1}}$ (recall Theorem~\ref{theo-link}) and  to change variables. This enables us to rephrase the question of whether all positive diagonal scalings of $H$ belong to $ \bigcup_r\MK^{(r)}_5$ as the question of deciding whether, for all positive scalars $d_1,\ldots,d_5$,  the form $(x^{\circ 2})^THx^{\circ 2}$ can be written as a sum of squares modulo the ideal generated by $\sum_{i=1}^5d_i x_i^2 -1$. This latter question was recently answered in the affirmative by Schweighofer and Vargas \cite{SV}.

\begin{theorem}[Schweighofer, Vargas \cite{SV}]\label{DHD-theorem}
Let $d_1, d_2, \dots, d_5>0$ be positive real numbers. Then we have
$$(x^{\circ 2})^T Hx^{\circ 2} = \sigma +q\Big(1-\sum_{i=1}^5 d_ix_i^2\Big) \ \text{ for some } \sigma\in \Sigma \text{ \rm and } q\in \oR[x].$$
Therefore, $DHD\in \bigcup_r \MK^{(r)}_5$ for all $D\in\MD^5_{++}$.
\end{theorem} 
The proof of this theorem uses the theory of pure states in real algebraic geometry (as described in \cite{BSS}), combined with a characterization of the diagonal scalings of the Horn matrix that belong to the cone $\MK_n^{(1)}$ (given in \cite{LV2021b}). The technical details go beyond the scope of this chapter, so we refer to \cite{SV} for details.

\ignore{
\begin{warning}{Horn Matrix $H$}
$H\notin \LAS_{\Delta_n}^{(r)}$ for any $r\geq 0$. The proof relies in exploiting the structure of the (infinitely many) zeros of $x^THx$ over $\Delta_n$.
\end{warning}
This example shows that another strategy should be applied
\\
We finish the section by showing that the cones $\LAS_{\Delta_n}^{(r)}$ approximate exactly $\COP_n$ if and only if $n\leq 3$.
}

\section{The stability number of a graph $\alpha(G)$}\label{section-alpha}

In this section, we investigate a class of copositive matrices that arise naturally from graphs.
Consider a graph $G=(V=[n],E)$, where $V=[n]$ is the set of vertices and $E$ is the set of edges, consisting of the pairs of distinct vertices that are adjacent in $G$. A set $S\subseteq V$ is called {\em stable} (or {\em independent}) if it does not contain any edge of $G$. Then,  the {\em stability number} of $G$, denoted by $\alpha(G)$, is defined as the maximum cardinality of a stable set in $G$. Computing $\alpha(G)$ is a well-known NP-hard problem (see \cite{Karp}), with many applications, e.g., in operations research, social networks analysis, and chemistry. There is a vast literature on this problem, dealing among other things with how to define linear and/or semidefinite approximations for $\alpha (G)$ (see, e.g., \cite{ dKP2002,Laurent2003, ZVP2006} and further references therein). 

\begin{example}{Lasserre hierarchy for $\alpha(G)$ via polynomial optimization on the binary cube}
The stability number of $G=([n],E)$ can be formulated as a polynomial optimization problem  on the binary cube $\{0,1\}^{n}$:
\begin{align}\label{alpha-binary}
\alpha(G)=\max\Big\{\sum_{i\in V}x_i:\  x_ix_j=0 \text{ for } \{i,j\}\in E,\ x_i^2-x_i=0 \text{ for } i\in V\Big\}.
\end{align}
We can consider the Lasserre hierarchy (\ref{lasserre-hierarchy}) for problem (\ref{alpha-binary}) and  obtain the following bounds 
\begin{align}\label{las-alpha}
\las^{(r)}(G):= \min \Big\{\lambda: \ & \lambda - \sum_{i\in V}x_i = \sigma + \sum_{\{i,j\}\in E}p_{ij}x_ix_j + \sum_{i\in V}q_i(x_i^2-x_i)\\
&\text{ for some  } \sigma\in \Sigma_{2r} \text{ and }   p_{ij}, q_i\in \oR[x]_{2r-2} \Big\}.
\end{align}
Clearly, we have $\alpha(G)\le \las^{(r)}(G)$. Moreover, the bound is exact at order $r=\alpha(G)$, that is,
$\alpha(G)=\las^{(\alpha(G))}(G)$ (see \cite{Laurent2003}). The proof is not difficult and exploits the fact that in the definition of these parameters one works modulo the ideal generated by the polynomials $x_i^2-x_i$ ($i\in V$) and the edge monomials $x_ix_j$ ($\{i,j\}\in E$).
At order $r=1$, the bound $\las^{(1)}(G)$ coincides with the parameter $\vartheta(G)$ introduced in 1979 by Lov\'asz in his seminal paper \cite{Lo79}.
\end{example}

In this section we focus on the hierarchies of approximations that naturally arise when considering the following copositive reformulation for $\alpha(G)$, given  by de Klerk and Pasechnik \cite{dKP2002}:
\begin{equation}\label{alpha-cop}
\alpha(G)=\min \{t:\  t(A_G+I)-J \in \COP_n\}.
\end{equation}
Here,  $A_G, I$, and $J$ are, respectively, the adjacency matrix of $G$ (whose entries are all 0 except 1 at the positions corresponding to the edges of $G$), the identity, and the all-ones matrix. 
As a consequence, it follows from (\ref{alpha-cop}) that the following {\em graph matrix}
\begin{equation}\label{matG}
M_G:=\alpha(G)(I+A_G)-J
\end{equation}
belongs to $\COP_n$. The copositive reformulation (\ref{alpha-cop}) for $\alpha(G)$ can be seen as an application of the following quadratic formulation by Motzkin and Straus \cite{motzkin}:
$$
{1\over \alpha(G)}=\min\{x^T(I+A_G)x: x\in \Delta_n\}.
$$

\begin{example}{
The Horn matrix coincides with the graph matrix of the graph $C_5$.}
When $G=C_5$ is the 5-cycle, its adjacency matrix $A_G$ is given by 
\begin{align*}\label{A_C_5}
A_{C_5}= \left(\begin{matrix} 0 & 1 & 0 & 0 & 1\cr
1 & 0 & 1 & 0 & 0\cr
0 & 1 & 0 & 1 & 0\cr
0 & 0 & 1 & 0 & 1 \cr
1 & 0 & 0 & 1 & 0
 \end{matrix}\right).
\end{align*}
As $\alpha(C_5)=2$, it follows that  the graph matrix $M_{C_5}=2(I+A_{C_5})-J$ of $C_5$ coincides with the Horn matrix $H$.
\end{example}
Based on the formulation (\ref{alpha-cop}), de Klerk and Pasechnik \cite{dKP2002} proposed two hierarchies $\zeta^{(r)}(G)$ and $\vartheta^{(r)}(G)$ of upper bounds for $\alpha(G)$,  that are obtained by replacing  in (\ref{alpha-cop}) the cone $\COP_n$ by its subcones $\MC_n^{(r)}$ and $\MK_n^{(r)}$,  respectively. In this section, we present several known results about these two hierarchies and related results for the graph matrices $M_G$.
One of the central questions  is whether the hierarchy $\vartheta^{(r)}(G)$ converges to $\alpha(G)$ in finitely many steps or, equivalently, whether the matrix $M_G$ belongs to $\bigcup_r\MK^{(r)}_n$, and what can be said about the minimum number of steps where finite convergence takes place.

\subsection{The hierarchy $\zeta^{(r)}(G)$}\label{sec-alpha-C}

As mentioned above, for an integer $r\ge 0$, the parameter $\zeta^{(r)}(G)$ is defined as
\begin{equation}
\zeta^{(r)}(G):=\min \{t:  \  t(A_G+I)-J \in \MC_n^{(r)}\}.
\end{equation}
Since $\text{int}(\COP_n)\subseteq  \bigcup_{r\geq 0}\MC_n^{(r)}$, it follows directly that the parameters  $\zeta^{(r)}(G)$ converge asymptotically  to $\alpha(G)$ as $r\to \infty$. Note that, if $G=K_n$ is a complete graph, then $\alpha(G)=1$ and the matrix $I+A_G-J$ is the zero matrix, thus belonging trivially to the cone $\MC^{(0)}_n$, so that $1=\alpha(K_n)=\zeta^{(0)}(K_n)$.  However, finite convergence does not hold if $G$ is not a complete graph. 
\begin{theorem}[de Klerk, Pasechnik \cite{dKP2002}]\label{theozeta}
Assume  $G$ is not a complete graph. Then, we have $\zeta^{(r)}(G)>\alpha(G)$ for all $r\in \mathbb{N}$.
\end{theorem}
By the definition of the cone $\MC^{(r)}_n$, the parameter $\zeta^{(r)}(G)$ can be formulated as a linear program, asking for the smallest scalar $t$ for which all the coefficients of the polynomial 
$(\sum_{i=1}^nx_i)^r \ x^T(t(I+A_G)-J)x$ are nonnegative.
The parameter $\zeta^{(r)}(G)$ is very well understood. Indeed, Pe\~{n}a, Vera and Zuluaga \cite{PVZ2007} give a closed-form expression for it in terms of $\alpha(G)$. 
\begin{theorem}[Pe\~{n}a, Vera, Zuluaga \cite{PVZ2007}]\label{theorem-zeta}
Write $r+2=u\alpha(G)+v$, where $u, v$ are nonnegative integers such that  $v\le \alpha(G)-1$. Then we have
$$ \zeta^{(r)}(G)=\frac{\binom{r+2}{2}}{\binom{u}{2}\alpha(G)+uv},$$
where we set $\zeta^{(r)}(G)=\infty$ if $r\le \alpha(G)-2$ (since then the denominator in the above formula is equal to 0).
\end{theorem}
So the above result shows that the bound $\zeta^{(r)}$ is useless for $r\leq \alpha(G)-2$. 
Another consequence is that after $r=\alpha(G)^2-1$ steps we find $\alpha(G)$ up to rounding.  (See also \cite{dKP2002} where this result is shown for $r=\alpha(G)^2$). 
\begin{corollary}[\cite{PVZ2007}]\label{corollary-zeta}
We have $\lfloor{\zeta^{(r)}(G)}\rfloor=\alpha(G)$ if and only if $r\geq \alpha(G)^2-1$.
\end{corollary}

\subsection{The hierarchy $\vartheta^{(r)}(G)$}\label{sec-alpha-K}

We now consider the parameter $\vartheta^{(r)}(G)$, for $r\in \oN$,  defined as follows in \cite{dKP2002}: 
\begin{equation}
\vartheta^{(r)}(G):=\min\{t:\  t(A_G+I)-J \in \MK_n^{(r)}\}.
\end{equation}
Since $\MC_n^{(r)}\subseteq \MK_n^{(r)}\subseteq \COP_n$ we have $\alpha(G)\leq \vartheta^{(r)}(G)\leq\zeta^{(r)}(G)$ for any $r\ge 0$, and thus 
 the parameters $\vartheta^{(r)}(G)$ converge asymptotically  to $\alpha(G)$ as $r\to \infty$. 

At order $r=0$, while the parameter $\zeta^{(0)}(G)=\infty$ is useless, the parameter $\vartheta^{(0)}(G)$ provides a useful bound for $\alpha(G)$. Indeed, it is shown in \cite{dKP2002} that $\vartheta^{(0)}(G)$ coincides with the  variation $\vartheta'(G)$ of the Lov\'asz theta number $\vartheta(G)$ (obtained by adding some nonnegativity constraints); so we have  the inequalities $\alpha(G)\le \vartheta'(G)=\vartheta^{(0)}(G)\le \vartheta(G)$ (see \cite{Lo79,Sch79}). This connection in fact motivates the choice of the notation $\vartheta^{(r)}(G)$.
For instance, if $G$ is a perfect graph\footnote{A graph $G$ is called {\em perfect} if its clique number $\omega(G)$ coincides with its chromatic number $\chi(G)$, and the same holds for any induced subgraph $G'$ of $G$. Here $\omega(G)$ denotes the maximum cardinality of a clique (a set of pairwise adjacent vertices) in $G$ and $\chi(G)$ is the minimum number of colors that are needed to color the vertices of $G$ in such a way that adjacent vertices receive distinct colors. An induced subgraph $G'$ of $G$ is any subgraph of $G$ of the form $G'=G[U]$, obtained by selecting a subset $U\subseteq V$ and keeping only the edges of $G$ that are contained in $U$.}, then we have $\vartheta(G)=\vartheta^{(0)}(G)=\alpha(G)$ (see  \cite{GLS} for a broad exposition). We also have $\vartheta(C_5)=\vartheta^{(0)}(C_5)$ (note that $C_5$ is not a perfect graph since $\omega(C_5)=2<\chi(C_5)=3$). But there exist graphs for which $\alpha(G)=\vartheta^{(0)}(G)<\vartheta(G)$ (see, e.g., \cite{Best}).

In Theorem \ref{theozeta} we saw that the bounds $\zeta^{(r)}(G)$ are never exact. This raises naturally the question of whether the (stronger) bonds $\vartheta^{(r)}(G)$ may be exact.  Recall the definition of  the graph matrix $M_G=\alpha(G)(A_G+I)-J$ in (\ref{matG}), and define the associated polynomial $p_G:=(x^{\circ2})^TM_Gx^{\circ2}$. Then, for any $r\in \mathbb{N}$, we have
$$\vartheta^{(r)}(G)=\alpha(G)\ \Longleftrightarrow \ M_G\in \MK_n^{(r)}\  \Longleftrightarrow \Big(\sum_{i=1}^{n}x_i^2\Big)^rp_G\in \Sigma.$$
As $M_G$ is copositive the polynomial $p_G$ is globally nonnegative. The point however is that $p_G$ has zeros in $\mathbb{R}^n\setminus \{0\}$. In particular,  every stable set $S\subseteq V$ of cardinality $\alpha(G)$ provides a zero $x=\chi^S$. Thus the question of whether $p_G$ admits  a positivity certificate of the form $(\sum_{I=1}^nx_i^2)^rp_G\in \Sigma$ for some $r\in\oN$ (as in (\ref{cert-reznick})) is nontrivial. 
 In \cite{dKP2002} it was in fact conjectured that such a certificate exists at order $r=\alpha(G)-1$; in other words,  that the parameter $\vartheta^{(r)}(G)$ is exact at order $r=\alpha(G)-1$. 


\medskip\noindent
{\bf Conjecture 1 (de Klerk and Pasechnik \cite{dKP2002}})

{\em For any graph $G$, we have $\vartheta^{(\alpha(G)-1)}(G)=\alpha(G)$,  or, equivalently, we have $M_G~\in~\MK_n^{(\alpha(G)-1)}$.}

\medskip

\begin{trailer}{Comparison of the parameters $\vartheta^{(r)}(G)$ and $\las^{(r)}(G)$}
At the beginning of Section \ref{section-alpha} we introduced the parameters $\las^{(r)}(G)$.
In \cite{GL2007} it is shown that, for any integer $r\ge 1$,  a slight strengthening of the parameter $\las^{(r)}(G)$ (obtained by adding some nonnegativity constraints) is at least as good as the parameter $\vartheta^{(r-1)}(G)$. The bounds $\las^{(r)}(G)$ are known to converge to $\alpha(G)$ in $\alpha(G)$ steps, i.e., $\las^{(\alpha(G))}(G)=\alpha(G)$. Thus Conjecture~1 asks whether a similar property holds for the parameters $\vartheta^{(r)}(G)$.  
While the finite convergence property for the Lasserre-type bounds is relatively easy to prove (by exploiting the fact  that one works modulo the ideal generated by  $x_i^2-x_i$ for $i\in V$ and $x_ix_j$ for $\{i,j\}\in E$)), proving Conjecture~1 seems much more challenging.
\end{trailer}

 Conjecture~1 is known to hold for some graph classes. For instance, we saw above that it holds for perfect graphs (with $r=0$), but it also holds for odd cycles and their complements -- that are not perfect (with $r=1$, see \cite{dKP2002}). In \cite{GL2007} Conjecture~1  was shown to hold for all graphs $G$ with $\alpha(G)\leq 8$ (see also \cite{PVZ2007} for the case $\alpha(G)\leq 6$). In fact, a stronger result is shown there: the proof relies on a technical construction of matrices that permit to certify  membership of $M_G$ in the cones $\MQ^{(r)}_n$ (and thus in the cones $\MK_n^{(r)}$).
 
\begin{theorem}[Gvozdenovi\'{c}, Laurent \cite{GL2007}]\label{alpha-8}
Let $G$ be a graph with $\alpha(G)\leq 8$. Then we have $\vartheta^{(\alpha(G)-1)}(G)=\alpha(G)$, or, equivalently, $M_G\in \MK^{(\alpha(G)-1)}_n$.
\end{theorem}
Whether Conjecture~1 holds in general is still an open problem. However, a weaker form of it has been recently settled; namely finite convergence of the hierarchy $\vartheta^{(r)}(G)$ to $\alpha(G)$, or, equivalently, membership of the graph matrices $M_G$ in $\bigcup_r\MK^{(r)}_n$.

\begin{theorem}[Schweighofer, Vargas \cite{SV}]\label{finite-conv}
For any graph $G$, we have $\vartheta^{(r)}(G)=\alpha(G)$ for some $r\in \mathbb{N}$. Equivalently, we have $M_G\in\bigcup_r\MK^{(r)}_n$.
\end{theorem}

In what follows we discuss some of the ingredients that are used for the proof of this result.
Here too, we will use the fact that 
$\bigcup_r \LAS^{(r)}_{\Delta_n}\subseteq \bigcup_r \MK^{(r)}_n=\bigcup_r \LAS^{(r)}_{\mathbb S^{n-1}}$ (recall Theorem~\ref{theo-link}) and so we we will consider the quadratic form $x^T M_Gx$ instead of the quartic form $p_G=(x^{\circ 2})^TM_Gx^{\circ 2}$. 
Whether  the quadratic form $x^TM_Gx$ has finitely many zeros in the simplex plays an important role. 
We will  first discuss the case when there are finitely many zeros, in which case one can show a stronger result, namely membership of $M_G$ in $\bigcup_r\LAS^{(r)}_{\Delta_n}$ (see Theorem~\ref{finite-acritical}
 below). 

\medskip
As we will see in  Corollary \ref{corzerosMG} below, whether the number of zeros of $x^T M_Gx$ in $\Delta_n$ is finite   is directly related to the notion of critical edges in the graph $G$. We first introduce this graph notion.

\begin{example}{Critical edges}
Let $G=(V,E)$ be a graph. The edge $e\in E$ is \textit{critical} is $\alpha(G\setminus e)=\alpha(G)+1$. Here $G\setminus e$ denotes the graph $(V,E\setminus\{e\})$.
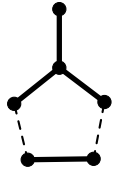
\begin{center}
	\definecolor{ududff}{rgb}{0.30196078431372547,0.30196078431372547,1.}
	\definecolor{f}{rgb}{0., 0., 0.}
	\begin{tikzpicture}[line cap=round,line join=round,=triangle 45,x=.6cm,y=.6cm]
	\clip(-1.44,1.34) rectangle (3.42,6.28);
	\draw [line width=2.pt] (-0.02,3.18)-- (0.98,3.98);
	\draw [line width=2.pt] (0.98,3.98)-- (1.98,3.22);
	\draw [thick, dashed] (1.74,1.96)-- (1.98,3.22);
	\draw [thick, dashed] (0.28,1.94)-- (-0.02,3.18);
	\draw [line width=2.pt] (0.28,1.94)-- (1.74,1.96);
	\draw [line width=2.pt] (0.98,3.98)-- (0.98,5.28);
	\begin{scriptsize}
	\draw [fill=f] (0.98,3.98) circle (2.5pt);
	\draw [fill=f] (-0.02,3.18) circle (2.5pt);
	\draw [fill=f] (0.28,1.94) circle (2.5pt);
	\draw [fill=f] (1.74,1.96) circle (2.5pt);
	\draw [fill=f] (1.98,3.22) circle (2.5pt);
	\draw [fill=f] (0.98,5.28) circle (2.5pt);
	\end{scriptsize}
	\end{tikzpicture}
	\end{center}
For example, for the above graph, the two dashed edges are its critical  edges.
	\end{example}
	
\begin{example}{Critical graphs}
We say that $G$ is \textit{critical} if all its edges are critical. For example,  odd cycles are critical graphs.
The next figure shows the 5-cycle $C_5$. \ignore{
\definecolor{uuuuuu}{rgb}{0.26666666666666666,0.26666666666666666,0.26666666666666666}
\definecolor{uququq}{rgb}{0.25098039215686274,0.25098039215686274,0.25098039215686274}
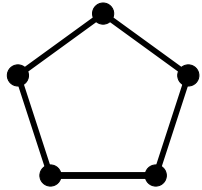
\begin{tikzpicture}[line cap=round,line join=round,=triangle 45,x=.3cm,y=.3cm]
\draw [line width=1pt] (-13.578509963461387,8.326106157908189)-- (-11.19999999999867,10.054194799648329);
\draw [line width=1pt] (-11.19999999999867,10.054194799648329)-- (-8.821490036535954,8.326106157908187);
\draw [line width=1pt] (-8.821490036535954,8.326106157908187)-- (-9.72999999999844,5.53);
\draw [line width=1pt] (-12.6699999999989,5.53)-- (-9.72999999999844,5.53);
\draw [line width=1pt] (-12.6699999999989,5.53)-- (-13.578509963461387,8.326106157908189);
\begin{scriptsize}
\draw [fill=uququq] (-12.6699999999989,5.53) circle (2.5pt);
\draw[color=uququq] (-13.66, 5.52);
\draw [fill=uququq] (-9.72999999999844,5.53) circle (2.5pt);
\draw[color=uququq] (-8.92, 5.52);
\draw [fill=uuuuuu] (-8.821490036535954,8.326106157908187) circle (2.5pt);
\draw[color=uququq] (-8.0, 8.52);
\draw [fill=uuuuuu] (-11.19999999999867,10.054194799648329) circle (2.5pt);
\draw[color=uququq] (-12, 10.5);
\draw [fill=uuuuuu] (-13.578509963461387,8.326106157908189) circle (2.5pt);
\draw[color=uququq] (-14.3, 8.9);
\end{scriptsize}
\end{tikzpicture}
}
\begin{center}
\definecolor{ududff}{rgb}{0.30196078431372547,0.30196078431372547,1}
\definecolor{uuuuuu}{rgb}{0.26666666666666666,0.26666666666666666,0.26666666666666666}
\definecolor{uququq}{rgb}{0.25098039215686274,0.25098039215686274,0.25098039215686274}
\begin{tikzpicture}[line cap=round,line join=round,=triangle 45,x=0.7cm,y=.7cm]
\clip(-8.774346433433802,2.2925404618158955) rectangle (-2.277383759104841,7.904409670366731);
\draw [line width=2.2pt] (-7.118033988749895,5.552113032590308)-- (-5.5,6.727683537175254);
\draw [line width=2.2pt] (-5.5,6.727683537175254)-- (-3.881966011250105,5.552113032590308);
\draw [line width=2.2pt] (-3.881966011250105,5.552113032590308)-- (-4.5,3.65);
\draw [line width=2.6pt] (-4.5,3.65)-- (-6.5,3.65);
\draw [line width=2.2pt] (-6.5,3.65)-- (-7.118033988749895,5.552113032590308);
\draw [line width=2.2pt] (-2.118033988749895,5.492113032590307)-- (-0.5,6.667683537175253);
\draw [line width=2.2pt] (-0.5,6.667683537175253)-- (1.118033988749895,5.492113032590306);
\draw [line width=2.2pt] (1.118033988749895,5.492113032590306)-- (0.5,3.59);
\draw [line width=2.2pt] (0.5,3.59)-- (-1.5,3.59);
\draw [line width=2.2pt] (-1.5,3.59)-- (-2.118033988749895,5.492113032590307);
\draw [line width=2.2pt,dash pattern=on 1pt off 1pt] (-0.5,6.667683537175253)-- (-0.5071968275485451,8.394035842924854);
\draw [line width=2.2pt] (-6.32622,0.9425110481835794)-- (-4.72551,0.9425110481835786);
\draw [line width=2.2pt] (-4.72551,0.9425110481835786)-- (-3.9251549999999993,-0.4437444759082114);
\draw [line width=2.2pt] (-3.9251549999999993,-0.4437444759082114)-- (-4.72551,-1.83);
\draw [line width=2.2pt] (-4.72551,-1.83)-- (-6.32622,-1.83);
\draw [line width=2.2pt] (-6.32622,-1.83)-- (-7.126575000000001,-0.4437444759082093);
\draw [line width=2.2pt] (-7.126575000000001,-0.4437444759082093)-- (-6.32622,0.9425110481835794);
\draw [line width=2.2pt] (-2.0190801509178193,-0.09932688610350526)-- (-0.40169272594544636,1.077133537175253);
\draw [line width=2.2pt] (-0.40169272594544636,1.077133537175253)-- (1.21698782658197,-0.09754704871588105);
\draw [line width=2.2pt] (1.21698782658197,-0.09754704871588105)-- (0.6,-2);
\draw [line width=2.2pt] (0.6,-2)-- (-1.4,-2.0011);
\draw [line width=2.2pt] (-1.4,-2.0011)-- (-2.0190801509178193,-0.09932688610350526);
\draw [line width=2.2pt] (0,-1.2)-- (-1.047213595499958,-0.439154786963877);
\draw [line width=2.2pt] (-1.047213595499958,-0.439154786963877)-- (0.24721359549995792,-0.4391547869638774);
\draw [line width=2.2pt] (0.24721359549995792,-0.4391547869638774)-- (-0.8,-1.2);
\draw [line width=2.2pt] (-0.8,-1.2)-- (-0.4,0.03107341487010118);
\draw [line width=2.2pt] (-0.4,0.03107341487010118)-- (0,-1.2);
\draw [line width=2.2pt] (-0.4,0.03107341487010118)-- (-0.40169272594544636,1.077133537175253);
\draw [line width=2.2pt] (0.24721359549995792,-0.4391547869638774)-- (1.21698782658197,-0.09754704871588105);
\draw [line width=2.2pt] (0,-1.2)-- (0.6,-2);
\draw [line width=2.2pt] (-0.8,-1.2)-- (-1.4,-2.0011);
\draw [line width=2.2pt] (-1.047213595499958,-0.439154786963877)-- (-2.0190801509178193,-0.09932688610350526);
\draw [line width=2pt] (-7,-3.5)-- (-3.878084707852556,-4.034936229704507);
\begin{scriptsize}
\draw [fill=black] (-6.5,3.65) circle (4pt);
\draw [fill=black] (-4.5,3.65) circle (4pt);
\draw [fill=black] (-3.881966011250105,5.552113032590308) circle (4pt);
\draw [fill=black] (-5.5,6.727683537175254) circle (4pt);
\draw [fill=black] (-7.118033988749895,5.552113032590308) circle (4pt);
\draw [fill=black] (-1.5,3.59) circle (4pt);
\draw [fill=black] (0.5,3.59) circle (4pt);
\draw [fill=black] (1.118033988749895,5.492113032590306) circle (4pt);
\end{scriptsize}
\end{tikzpicture}

\end{center}
\end{example}
\begin{example}{Acritical graphs}
We say that $G$ is \textit{acritical} if it does not have critical edges. Every even cycle is acritical, as well as the Petersen graph. The next figure shows the 6-cycle $C_6$ and the Petersen graph.
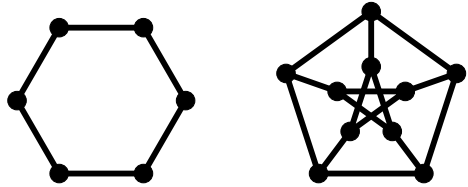
\begin{center}
\definecolor{ududff}{rgb}{0.30196078431372547,0.30196078431372547,1}
\definecolor{uuuuuu}{rgb}{0.26666666666666666,0.26666666666666666,0.26666666666666666}
\definecolor{uququq}{rgb}{0.25098039215686274,0.25098039215686274,0.25098039215686274}
\begin{tikzpicture}[line cap=round,line join=round,=triangle 45,x=.7cm,y=.7cm]
\clip(-7.722083293278789,-2.6468648622780937) rectangle (2.529540021760409,2.6920000915198554);
\draw [line width=2.2pt] (-7.118033988749895,5.552113032590308)-- (-5.5,6.727683537175254);
\draw [line width=2.2pt] (-5.5,6.727683537175254)-- (-3.881966011250105,5.552113032590308);
\draw [line width=2.2pt] (-3.881966011250105,5.552113032590308)-- (-4.5,3.65);
\draw [line width=2.6pt] (-4.5,3.65)-- (-6.5,3.65);
\draw [line width=2.2pt] (-6.5,3.65)-- (-7.118033988749895,5.552113032590308);
\draw [line width=2.2pt] (-2.118033988749895,5.492113032590307)-- (-0.5,6.667683537175253);
\draw [line width=2.2pt] (-0.5,6.667683537175253)-- (1.118033988749895,5.492113032590306);
\draw [line width=2.2pt] (1.118033988749895,5.492113032590306)-- (0.5,3.59);
\draw [line width=2.2pt] (0.5,3.59)-- (-1.5,3.59);
\draw [line width=2.2pt] (-1.5,3.59)-- (-2.118033988749895,5.492113032590307);
\draw [line width=2.2pt,dash pattern=on 1pt off 1pt] (-0.5,6.667683537175253)-- (-0.5071968275485451,8.394035842924854);
\draw [line width=2.2pt] (-6.32622,0.7725110481835795)-- (-4.72551,0.7725110481835786);
\draw [line width=2.2pt] (-4.72551,0.7725110481835786)-- (-3.9251549999999993,-0.6137444759082114);
\draw [line width=2.2pt] (-3.9251549999999993,-0.6137444759082114)-- (-4.72551,-2);
\draw [line width=2.2pt] (-4.72551,-2)-- (-6.32622,-2);
\draw [line width=2.2pt] (-6.32622,-2)-- (-7.126575000000001,-0.6137444759082092);
\draw [line width=2.2pt] (-7.126575000000001,-0.6137444759082092)-- (-6.32622,0.7725110481835795);
\draw [line width=2.2pt] (-2.018033988749895,-0.09788696740969272)-- (-0.4,1.0776835371752531);
\draw [line width=2.2pt] (-0.4,1.0776835371752531)-- (1.218033988749895,-0.09788696740969349);
\draw [line width=2.2pt] (1.218033988749895,-0.09788696740969349)-- (0.6,-2);
\draw [line width=2.2pt] (0.6,-2)-- (-1.4,-2);
\draw [line width=2.2pt] (-1.4,-2)-- (-2.018033988749895,-0.09788696740969272);
\draw [line width=2.2pt] (0,-1.2)-- (-1.047213595499958,-0.439154786963877);
\draw [line width=2.2pt] (-1.047213595499958,-0.439154786963877)-- (0.24721359549995792,-0.4391547869638774);
\draw [line width=2.2pt] (0.24721359549995792,-0.4391547869638774)-- (-0.8,-1.2);
\draw [line width=2.2pt] (-0.8,-1.2)-- (-0.4,0.03107341487010118);
\draw [line width=2.2pt] (-0.4,0.03107341487010118)-- (0,-1.2);
\draw [line width=2.2pt] (-0.4,0.03107341487010118)-- (-0.4,1.0776835371752531);
\draw [line width=2.2pt] (0.24721359549995792,-0.4391547869638774)-- (1.218033988749895,-0.09788696740969349);
\draw [line width=2.2pt] (0,-1.2)-- (0.6,-2);
\draw [line width=2.2pt] (-0.8,-1.2)-- (-1.4,-2);
\draw [line width=2.2pt] (-1.047213595499958,-0.439154786963877)-- (-2.018033988749895,-0.09788696740969272);
\draw [line width=2pt] (-7,-3.5)-- (-3.878084707852556,-4.034936229704507);
\begin{scriptsize}
\draw [fill=uququq] (-6.5,3.65) circle (3.5pt);
\draw [fill=uququq] (-4.5,3.65) circle (3.5pt);
\draw [fill=uuuuuu] (-3.881966011250105,5.552113032590308) circle (3.5pt);
\draw [fill=uuuuuu] (-5.5,6.727683537175254) circle (3.5pt);
\draw [fill=uuuuuu] (-7.118033988749895,5.552113032590308) circle (3.5pt);
\draw [fill=uququq] (-1.5,3.59) circle (3.5pt);
\draw [fill=uququq] (0.5,3.59) circle (3.5pt);
\draw [fill=uuuuuu] (1.118033988749895,5.492113032590306) circle (3.5pt);
\draw [fill=uuuuuu] (-0.5,6.667683537175253) circle (3.5pt);
\draw [fill=uuuuuu] (-2.118033988749895,5.492113032590307) circle (3.5pt);
\draw [fill=uququq] (-0.5071968275485451,8.394035842924854) circle (3.5pt);
\draw [fill=black] (-6.32622,-2) circle (3.5pt);
\draw [fill=black] (-4.72551,-2) circle (3.5pt);
\draw [fill=black] (-3.9251549999999993,-0.6137444759082114) circle (3.5pt);
\draw [fill=black] (-4.72551,0.7725110481835786) circle (3.5pt);
\draw [fill=black] (-6.32622,0.7725110481835795) circle (3.5pt);
\draw [fill=black] (-7.126575000000001,-0.6137444759082092) circle (3.5pt);
\draw [fill=black] (-1.4,-2) circle (3.5pt);
\draw [fill=black] (0.6,-2) circle (3.5pt);
\draw [fill=black] (1.218033988749895,-0.09788696740969349) circle (3.5pt);
\draw [fill=black] (-0.4,1.0776835371752531) circle (3.5pt);
\draw [fill=black] (-2.018033988749895,-0.09788696740969272) circle (3.5pt);
\draw [fill=black] (-0.8,-1.2) circle (3.5pt);
\draw [fill=black] (0,-1.2) circle (3.5pt);
\draw [fill=black] (0.24721359549995792,-0.4391547869638774) circle (3.5pt);
\draw [fill=black] (-0.4,0.03107341487010118) circle (3.5pt);
\draw [fill=black] (-1.047213595499958,-0.439154786963877) circle (3.5pt);
\draw [fill=ududff] (-7,-5.5) circle (2.5pt);
\draw [fill=ududff] (-5,-5.5) circle (2.5pt);
\draw [fill=uuuuuu] (-5,-3.5) circle (2.5pt);
\draw [fill=uuuuuu] (-7,-3.5) circle (2.5pt);
\draw [fill=ududff] (-3,-5.5) circle (2.5pt);
\draw [fill=ududff] (-3,-3.5) circle (2.5pt);
\draw [fill=ududff] (-3.878084707852556,-4.034936229704507) circle (2.5pt);
\draw[color=ududff] (-2.483478714200079,-3.38628512583227) node {$K_1$};
\draw[color=black] (-5.3408993936975575,-3.7873266247091113) node {$l_3$};
\end{scriptsize}
\end{tikzpicture}
\end{center}
\end{example}

We now explain the role played by the critical edges in the description of the zeros of the form $x^TM_Gx$ in the simplex $\Delta_n$. First, note that, if $S$ is a stable set of size $\alpha(G)$, then $x=\chi^S/|S|$ is a zero. However, in general, there are more zeros. A characterization of the zeros was given in \cite{LV2021a} (see also \cite{GHPR}).

\begin{theorem}[\cite{LV2021a}]\label{minimizers-ms} 
Let $x\in \Delta_n$ with support $S:=\{i\in V:x_i>0\}$ and let $V_1, V_2, \dots, V_k$ denote the connected components of  $G[S]$, the subgraph of $G$ induced by the support $S$ of $x$. Then $x$ is a zero of the form $x^TM_Gx$ if and only if $k=\alpha(G)$ and, for all $h\in [k]$, $V_h$ is a clique of $G$ and $\sum_{i\in V_h}x_i=\frac{1}{\alpha(G)}$. In addition, the edges that are contained in $S$ 
 are critical edges of $G$.
\end{theorem}
In particular, we can characterize  the graphs $G$ for which the form $x^TM_Gx$ has  finitely many zeros in $\Delta_n$. 
\begin{corollary}[\cite{LV2021a}]\label{corzerosMG}
Let $G$ be a graph. The form $x^TM_Gx$ has finitely many zeros in $\Delta_n$ if and only if $G$ is acritical (i.e., $G$ has no critical edge). In that case, the zeros are the vectors of the form $\chi^S/|S|$, where $S$ is a stable set of size $\alpha(G)$.
\end{corollary}

\begin{example}{Zeros of the form $x^TM_Gx$ for the cycles $C_4$ and $C_5$}
The 4-cycle $C_4$ has vertex set $\{1,2,3,4\}$ and edges $\{1,2\}$, $\{2,3\}$, $\{3,4\}$, and $\{4,1\}$. It  has stability number $\alpha(C_4)=2$, it is acritical, and its maximum stable sets are the sets $\{1,3\}$ and $\{2,4\}$. Then, in view of Corollary \ref{corzerosMG}, the only zeros of the form $x^TM_{C_4}x$ in $\Delta_4$ are $(\frac{1}{2},0,\frac{1}{2}, 0)$ and $(0, \frac{1}{2}, 0, \frac{1}{2})$. 

The 5-cycle $C_5$ has vertex set $\{1,2,3,4,5\}$ and edges $\{1,2\}$, $\{2,3\}$, $\{3,4\}$, $\{4, 5\}$, and $\{5,1\}$. It has stability number $\alpha(C_5)=2$ and it is critical. Then, in view of Theorem \ref{minimizers-ms}, the form $x^TM_{C_5}x$ has infinitely many zeros in $\Delta_5$. For example, for any $t\in (0,1)$, the point $x_t=(\frac{1}{2}, 0, \frac{t}{2}, \frac{1-t}{2},0)$ is a zero supported in the two cliques $\{1\}$ and $\{3,4\}$ (indeed a critical edge). It can be checked that (up to symmetry)  all zeros take the shape of $x_t$ for $t\in [0,1]$.
\end{example}

When $G$ is an acritical graph one can show that its graph matrix $M_G$ belongs to one of the 
cones $\LAS_{\Delta_n}^{(r)}$, thus a stronger result than the result  from Theorem \ref{finite-conv}.

\begin{theorem}[Laurent, Vargas \cite{LV2021a}]\label{finite-acritical}
Let $G$ be an acritical graph. Then we have $M_G\in \bigcup_{r\geq 0}\LAS_{\Delta_n}^{(r)}$.
\end{theorem} 
As $\LAS_{\Delta_n}^{(r)}\subseteq \MK_n^{(r)}$ for any $r\in\oN$, this result implies  finite convergence of the hierarchy of bounds $\vartheta^{(r)}(G)$ to $\alpha(G)$ for the class of acritical graphs.  

The proof of Theorem \ref{finite-acritical}  relies on applying Theorem \ref{theo-Nie}.  By assumption, $G$ is acritical, and thus the quadratic form $x^TM_Gx$ has finitely many zeros in $\Delta_n$, as described in Corollary \ref{corzerosMG}. Now it suffices to verify that the zeros satisfy the conditions of Theorem \ref{theo-Nie}. We next  give the (easy) details for the sake of concreteness.

\begin{lemma}[\cite{LV2021a}]\label{lemacritical}
Let $G$ be an acritical graph and let $S$ be a stable set of size $\alpha(G)$. Then, for $x=\chi^S/\alpha(G)$, we have $(M_Gx)_i>0$ for $i\notin S$.   
\end{lemma} 
\begin{proof}
For a vertex $i\in V\setminus S$, let $N_S(i)$ denote the number of neighbours of $i$ in $S$. We have $N_S(i)\geq 1$ because $S\cup\{i\}$ is not stable, as $S$ is a stable set of size $\alpha(G)$. Since $G$ is acritical we must have $N_S(i)\geq 2$. Indeed, if $N_S(i)=1$ and $j\in S$ is the only neighbour of $i$ in $S$, then $\{i,j\}$ is a critical edge, contradicting the assumption on $G$.  Now we compute $(M_Gx)_i$:
\begin{align*}
 (M_Gx)_i&= \frac{1}{\alpha(G)}((\alpha(G)-1)N_S(i) - (\alpha(G)-N_S(i)))\\
 		&=\frac{1}{\alpha(G)}(\alpha(G)N_S(i)-\alpha(G))>0,
 \end{align*}
 where the last inequality holds as $N_S(i)\geq 2$.
\end{proof}
The above strategy does not  extend for general graphs (having some critical edges) and also the result of Theorem~\ref{finite-acritical} does not extend. For example, if $G=C_5$ is the 5-cycle (whose edges are all critical), then $M_G$ is the Horn matrix that does not belong to any of the 
cones  $\LAS_{\Delta_n}^{(r)}$ (as we saw in Section~\ref{section-membership-LAS}). 
Hence  another strategy is needed to show membership of $M_G$ in $\bigcup_r\MK^{(r)}_n$ for general graphs.
We now sketch some of the key ingredients that are used to show this result.

\subsubsection*{Some key ingredients for the proof for Theorem \ref{finite-conv}}For studying Conjecture~1 and, in general, the membership of the graph matrices $M_G$ in the cones $\MK_n^{(r)}$, it turns out that  the graph notion of  isolated nodes plays a crucial role. 

A node $i$ of a graph $G$ is said to be an {\em isolated node} of $G$ if $i$ is not adjacent to any other node of $G$.
Given a graph $G=(V,E)$ and a new node $i_0\not\in V$, the graph $G\oplus i_0$ is the graph $(V\cup\{i_0\},E)$ obtained by adding $i_0$ as an isolated node to $G$. The following result makes the link to Conjecture~1 clear.

\begin{theorem}[Gvozdenovi\'{c}, Laurent \cite{GL2007}]\label{theoGL}
Assume that, for any graph $G=([n],E)$ and $r\in\oN$,  we have
\begin{align}\label{eqimpli}
M_G\in  \MK^{(r)}_n \Longrightarrow M_{G\oplus i_0}\in \MK^{(r)}_{n+1}.
\end{align}
Then Conjecture~1 holds.
\end{theorem}
 Moreover,  it was conjectured in \cite{GL2007} that (\ref{eqimpli}) holds for each $r\in\oN$
  (which, if true,  would thus imply Conjecture~1).
 However, this conjecture was disproved in \cite{LV2021a}. 
 
 \begin{example}{Adding an isolated node may not preserve membership in $\MK^{(r)}$}
 Consider the 5-cycle $C_5$, whose graph matrix coincides with the Hall matrix: $M_{C_5}=H$.
 As we have seen earlier, $M_{C_5}\in \MK^{(1)}_5$. In \cite{LV2021a} it is shown that, if $G=C_5\oplus i_1\oplus \dots \oplus i_8$ is the graph obtained by adding eight isolated nodes to the 5-cycle, then $M_G\in \MK_{13}^{(1)}$,  but, if we add one more isolated node $i_0$ to $G$ (thus we add nine isolated nodes to $C_5$), then we have $M_{G\oplus i_0}\notin \MK_{14}^{(1)}$. 
  \end{example}

Hence, one cannot rely on the result of Theorem \ref{theoGL} and a new strategy is needed for solving Conjecture~1. 
The following variation of Theorem \ref{theoGL} is shown in \cite{LV2021b}, which can serve as a basis for proving a weaker form of Conjecture~1, namely membership of $M_G$ in $\bigcup_r \MK^{(r)}_n$.

\begin{theorem}[Laurent and Vargas \cite{LV2021b}]\label{equiv}
The following two assertions are equivalent.
\begin{description}
\item[(i)] For any graph $G=([n],E)$,  $M_G\in \bigcup_{r\geq0}\MK_{n}^{(r)}$ implies $M_{G\oplus i_0}\in \bigcup_{r\geq0}\MK_{n+1}^{(r)}$.
\item[(ii)] For any graph $G=([n],E), $ we have $M_G\in\bigcup_{r\geq0}\MK_n^{(r)}$.
\end{description}
\end{theorem}
This result is used as a crucial ingredient  in \cite{SV} for showing Theorem \ref{finite-conv}; namely, the authors of \cite{SV} show that Theorem \ref{equiv}~(i) holds.
The starting point of their proof is to use the fact that  $\bigcup_{r\geq0}\MK_n^{(r)}=\bigcup_{r\geq0}\LAS^{(r)}_{\mathbb{S}^{n-1}}$ (by Theorem \ref{theo-link}) and then to show that membership of the graph matrices in  $\bigcup_{r\geq0}\LAS^{(r)}_{\mathbb{S}^{n-1}}$ is preserved after adding isolated nodes. Recall that $p_G=(x^{\circ2})^TM_Gx^{\circ2}=\sum_{i,j\in V}x_i^2x_j^2(M_G)_{ij}.$

\begin{theorem}[Schweighofer and Vargas \cite{SV}]\label{isolated-preserve} 
Let $G=([n],E)$ be a graph. Assume that $p_G=\sigma_0 + q(\sum_{i=1}^nx_i^2-1)$ for some $\sigma_0\in \Sigma$ and $q_0\in \mathbb{R}[x_1, \dots, x_n]$. Then $p_{G\oplus i_0}=\sigma_1+q_1(x_{i_0}^2+\sum_{i=1}^nx_i^2-1)$ for some $\sigma_1\in \Sigma$ and $q_1\in \mathbb{R}[x_{i_0}, x_1, \dots, x_n]$.
\end{theorem}
Here too, the proof of this theorem uses the theory of pure states in real algebraic geometry (as described in \cite{BSS}).
 The technical details are too involved and thus go beyond the scope of this chapter,  we refer to \cite{SV} for the full details. As explained above, this theorem  implies Theorem \ref{finite-conv}.  The result (and proof) of Theorem \ref{isolated-preserve}, however, does not give any explicit bound on the degree of $\sigma_1$ in terms of the degree of $\sigma_0$. Hence one cannot infer any information on the degree of a representation of $p_G$ in $\Sigma+ I(\sum_{i=1}^nx_i^2-1)$. In other words, this result gives no information on the number of steps at which finite convergence of $\vartheta^{(r)}(G)$ to $\alpha(G)$ takes place.

\smallskip
Therefore, the status of Conjecture~1 remains widely open and its resolution likely requires new techniques. There is some evidence for its validity; for instance, Conjecture~1 holds for perfect graphs and for graphs $G$ with $\alpha(G)\le 8$ (Theorem~\ref{theoGL}), and any graph matrix $M_G$ belongs to some cone $\MK^{(r)}_n$ (Theorem \ref{finite-conv}). These facts also make the search for a possible counterexample a rather difficult task. 

\section{Concluding remarks}\label{secfinal}
In this chapter we have discussed several hierarchies of conic inner approximations for the copositive cone $\COP_n$, motivated by various sum-of-squares  certificates for positive polynomials on $\oR^n$, $\oR^n_+$, the simplex $\Delta_n$, and the unit sphere $\mathbb S^{n-1}$. 
The main players are Parrilo's  cones $\MK^{(r)}_n$, originally defined as the sets of matrices $M$ for which the polynomial $(\sum_{i=1}^nx_i^2)^r (x^{\circ 2})^TMx^{\circ 2}$ is a sum of squares of polynomials, thus  having a  certificate ``with denominator" (for positivity on $\oR^n$). The question whether these cones cover the full copositive cone is completely settled: the answer is positive for $n\le 5$ and negative for $n\ge 6$. The cones $\MK^{(r)}_n$ also capture the class of copositive graph matrices, of the form $M_G=\alpha(G)(A_G+I)-J$ for some graph~$G$.  
The challenge in settling these questions lies in the fact that, for any copositive matrix lying on the border of $\COP_n$, the associated  form has (nontrivial) zeros (and thus is not {\em strictly} positive), so that the classical positivity certificates do not suffice to claim membership in the conic approximations, and thus other techniques are needed.

A useful step is understanding the links to  other certificates ``without denominators" for positivity on the simplex or the sphere, which lead to the Lasserre-type cones $\LAS^{(r)}_{\Delta_n}$ and $\LAS^{(r)}_{\mathbb S^{n-1}}$.  Roughly speaking, the simplex-based cones form a weaker hierarchy, while the sphere-based cones provide an equivalent formulation for Parrilo's cones (see Theorem \ref{theo-link} and relation (\ref{link-rep}) for the exact relationships). Membership in the simplex-based cones can be shown for some  classes of copositive matrices, which thus implies membership in Parrilo's cones. 

We recall Conjecture 1 that asks whether any graph matrix $M_G$ belongs to the cone $\MK^{(r)}_n$ of order $r=\alpha(G)-1$, still widely open for graphs with $\alpha(G)\ge 9$.  The resolution of Conjecture~1 would offer an interesting  result that is relevant to the intersection of combinatorial optimization (about the computation of $\alpha(G)$), matrix copositivity (membership of a class of structured copositive matrices in one of Parrilo's approximation cones),   and real algebraic geometry (a sum-of-squares representation result with an explicit degree bound for a polynomial with zeros).

Matrix copositivity revolves around the question of deciding   whether a quadratic form  is nonnegative on $\oR^n_+$. This fits, more generally, within the study of copositive tensors, thus going from quadratic forms to forms with degree $d\ge 2$. There is a wide literature on  copositive tensors; we refer, e.g., to \cite{NieYangZhang,Qi2013,SongQi} and further references therein. 
 The relationships between the various types of positivity certificates discussed in this chapter  for the case $d=2$ 
extend to the case $d\ge 2$. (Note indeed that Theorems \ref{theoZVP} and \ref{dKPL-prop} hold for general homogeneous polynomials.) An interesting research direction may be to understand classes of structured symmetric tensors that are captured by some of the corresponding conic hierarchies.

\smallskip\noindent

\end{document}